\documentclass[12pt,twoside]{amsart}
\usepackage{latexsym,amsfonts,amsmath,amssymb}
\usepackage{enumerate,soul}
\usepackage[usenames, dvipsnames]{color}

\newtheorem{Th}{Theorem}[section]
\newtheorem{Rem}[Th]{Remark}
\newtheorem{Ex}[Th]{Example}
\newtheorem{Lemma}[Th]{Lemma}
\newtheorem{Def}[Th]{Definition}
\newtheorem{Prop}[Th]{Proposition}
\newtheorem{Cor}[Th]{Corollary}

\catcode`\@=11

\renewcommand{\section}%
   {\setcounter{equation}{0}\@startsection {section}{1}{\z@}{-3.5ex plus -1ex
  minus -.2ex}{2.3ex plus .2ex}{\Large\bf}}

\def\supp{\mathop{\rm supp}\nolimits}

\def\WF{\mathop{\rm WF}\nolimits}

\def\esssup{\mathop{\rm ess\,sup}}
\def\conesupp{\mathop{\rm cone\,supp}}
\def\ds{\displaystyle}
\def\R{\mathbb R}
\def\C{\mathbb C}
\def\N{\mathbb N}

\def\Z{\mathbb Z}

\newcommand{\D}{\mathcal{D}}

\newcommand{\F}{\mathcal{F}}

\newcommand{\Sch}{\mathcal{S}}

\newcommand{\afrac}[2]{\genfrac{}{}{0pt}{1}{#1}{#2}}

\newcommand{\beqsn}{\arraycolsep1.5pt\begin{eqnarray*}}
\newcommand{\eeqsn}{\end{eqnarray*}\arraycolsep5pt}
\newcommand{\beqs}{\arraycolsep1.5pt\begin{eqnarray}}
\newcommand{\eeqs}{\end{eqnarray}\arraycolsep5pt}

\title{The Gabor wave front set in spaces of ultradifferentiable functions}

\author[Boiti]{Chiara Boiti}
\address{
Dipartimento di Matematica e Informatica \\Universit\`a di Ferrara\\
Via Ma\-chia\-vel\-li n.~30\\
I-44121 Ferrara\\
Italy}
\email{chiara.boiti@unife.it}

\author[Jornet]{David Jornet}
\address{
Instituto Universitario de Matem\'atica Pura y Aplicada IUMPA\\
Universitat Po\-li\-t\`ecni\-ca de Val\`encia\\
Camino de Vera, s/n\\
E-46071 Valencia\\
Spain}
\email{djornet@mat.upv.es}

\author[Oliaro]{Alessandro Oliaro}
\address{Dipartimento di Matematica\\ Universit\`a di Torino\\
  Via Carlo Alberto n. 10\\ I-10123 Torino\\ Italy}
  \email{alessandro.oliaro@unito.it}

\begin{document}


\begin{abstract}
Given a non-quasianalytic subadditive weight function $\omega$ we consider the weighted
Schwartz space $\Sch_\omega$ and  the short-time Fourier transform on
$\Sch_\omega$, $\Sch'_\omega$ and on the related modulation spaces with
exponential weights.
In this setting we define the $\omega$-wave front set $\WF'_\omega(u)$ and the
Gabor $\omega$-wave front set $\WF^G_\omega(u)$ of $u\in\Sch'_{\omega}$, and we prove that they coincide.
Finally we look at applications of this wave front set for operators of differential and pseudo-differential type.
\end{abstract}

\maketitle
\markboth{\sc The Gabor wave front set in spaces of ultradifferentiable functions}
{\sc C.~Boiti, D.~Jornet and A.~Oliaro}

\section{Introduction}
The wave front set is a basic concept in the theory of linear partial differential operators. It deals with the analysis of singularities of a function (or distribution), and in the classical Schwartz distributions theory it was originally defined in \cite{H}. The idea is that, if a distribution $u\in\D^\prime(\R^d)$ coincides with a $C^\infty$ function in a neighborhood of a certain point $x_0\in\R^d$, then there exists a cut-off function $\varphi\in \D(\R^d)$ (i.e., $\varphi\in C^\infty(\R^d)$ with compact support) such that $\varphi u\in \D(\R^d)$, and consequently $\widehat{\varphi u}$ is a rapidly decreasing function. If $u$ is not $C^\infty$ at $x_0$, then $\widehat{\varphi u}$ does not decrease rapidly at least in some directions, and these directions are responsible for the absence of regularity of $u$ at $x_0$. The wave front set collects all the points $(x_0,\xi_0)$, with $\xi_0\neq 0$, where, roughly speaking, the distribution $u$ is not $C^\infty$ at $x_0$ due (on the Fourier transform side) to the absence of rapid decreasing in the direction $\xi_0$. The wave front set is then a subset of $\R^d\times(\R^d\setminus\{0\})$, is computed for a distribution $u$ and has to do with the analysis of the points where $u$ is not smooth in connection with the directions where absence of smoothness is shown by the Fourier transform of $u$. In the  case described above, the distribution space is $\D^\prime$, and `smooth' means $C^\infty$. 

Some very natural questions arise. First, the spaces $\D^\prime$ and $C^\infty$ could be replaced by other spaces of distributions connected to other concepts of `smoothness'; in this frame we refer for example to \cite{R}, \cite{BJJ}, \cite{AJO1}, \cite{AJO2}, where wave front sets connected with Gevrey and ultradifferentiable type regularity are considered. Moreover, in the case of $\D^\prime$ and $C^\infty$ the smoothness is intended in a local sense, but also some kind of global regularity could be considered, see for example \cite{H2}, \cite{CS}, \cite{RW}, \cite{N}, \cite{SW2}.

The wave front set, moreover, has very important applications in the study of propagation of singularities for partial differential (or more generally for pseudodifferential) operators. In this frame, different classes of pseudodifferential operators lead to corresponding variants of the wave front set, adapted to the class under consideration. Among the vast literature in this field we refer for example to \cite{H3} for the $C^\infty$ case, to \cite{R} for the Gevrey case, and to \cite{SW1}, \cite{CW} for the case of global wave front set defined in the spirit of the present paper.


Time-frequency analysis is a field of research that in the last decades has had a very big growth, with the development of many new techniques. One of the basic ideas of time-frequency analysis is the simultaneous analysis of a function (or distribution) with respect to variables and covariables, in order to quantify the energy of a signal at some time $x_0$ and some frequency $\xi_0$. Since the wave front set has to do with a simultaneous analysis of points (variables) and directions (covariables), it is very natural to try to apply methods of time-frequency analysis in connection with the wave front set. The work \cite{RW} is a very interesting contribution in this direction. 

The present paper deals with a global wave front set, in the spirit of \cite{H2}, treated with techniques from time-frequency analysis, following ideas from \cite{RW}. In particular, we study the case of ultradifferentiable functions in the sense of \cite{B}, \cite{BMT}, focusing on the space $\Sch_\omega(\R^d)$, defined as the space of functions $f\in\Sch(\R^d)$ such that for every $\lambda>0$ and $\alpha\in\N^d_0$ we have
$$
\sup_{x\in\R^d}e^{\lambda\omega(x)} \vert \partial^\alpha f(x)\vert < +\infty \ ,\qquad \sup_{\xi\in\R^d}e^{\lambda\omega(\xi)}\vert \partial^\alpha\hat{f}(\xi)\vert < +\infty ,
$$
where $\omega$ is a non-quasianalytic subadditive weight, cf.\ Definition \ref{defweight}. The spaces $\Sch_\omega(\R^d)$, together with the corresponding (ultra)distribution space $\Sch^\prime_\omega(\R^d)$, play a role similar to $\Sch(\R^d)$ and $\Sch^\prime(\R^d)$ in the classical Schwartz frame. We study in this paper a global wave front set adapted to $\Sch_\omega$ regularity, giving two different definitions; one is based on Gabor transform and the other is related to Gabor frames. We show that these definitions are equivalent. Moreover, we give applications of this global wave front set to pseudodifferential operators, to partial differential operators with polynomial coefficients, to localization operators, and we analyze some examples. The techniques are related with time-frequency analysis, in particular to Gabor transform, Gabor frames and modulation spaces. 

The paper is organized as follows. In Section \ref{sec2} we present basic definitions and consider time-frequency analysis on ultradifferentiable spaces;  we revise some known properties and prove other results that are needed in the paper, but that we could not find in the literature. In Section~\ref{sec3} we  give the two definitions of wave front set and prove that they are equivalent. Moreover we show that the global wave front set of a distribution $u\in\Sch^\prime_\omega(\R^d)$ is empty if and only if $u\in\Sch_\omega(\R^d)$, and that the global wave front set is not affected by translations and modulations, that are the basic operators of time-frequency analysis. In Section \ref{sec4} we present some applications about the action of operators of differential and pseudodifferential type on the wave front set, and finally in Section \ref{sec5} we analyze some examples.
\section{Preliminaries and the short-time Fourier transform in $\Sch_{\omega}(\R^{d})$}
\label{sec2}
Given a function $f\in L^1(\R^d)$, the Fourier transform of $f$ is defined as
\beqsn
\F(f)=
\hat{f}(\xi)=\int_{\R^d} e^{-i\langle x,\xi\rangle} f(x)\,dx,
\eeqsn
with standard extensions to more general spaces of functions and distributions.

\begin{Def}
\label{defweight}
A {\em non-quasianalytic subadditive weight function} is a continuous
increasing function
$\omega: \ [0,+\infty)\to[0,+\infty)$ satisfying the following properties:
\begin{itemize}
\item[$(\alpha)$]
  $\omega(t_1+t_2)\leq\omega(t_1)+\omega(t_2)\quad\forall t_1,t_2\geq0$;
\item[$(\beta)$]
  $\ds\int_1^{+\infty}\frac{\omega(t)}{t^2}dt<+\infty$;
\item[$(\gamma)$]
  $\log t=o(\omega(t))$ when $t\to+\infty$;
\item[$(\delta)$]
  $\varphi_\omega(t):=\omega(e^t)$ is convex.
\end{itemize}
We then define $\omega(\zeta):=\omega(|\zeta|)$ for $\zeta\in\C^d$.
\end{Def}

We denote by $\varphi^*_\omega$ the Young conjugate of $\varphi_\omega$, defined by
\beqsn
\varphi^*_\omega(s):=\sup_{t\geq0}\{ts-\varphi_\omega(t)\}
\eeqsn
and recall that $\varphi^*_\omega$ is convex and increasing,
$\varphi^{**}_\omega=\varphi_\omega$ and $\varphi^*_\omega(s)/s$ is increasing
(up to assume, without any loss of generality, that $\omega|_{[0,1]}\equiv0$).

\begin{Def}
  \label{defSomega}
  We define $\Sch_\omega(\R^d)$ as the set of all $u\in\Sch(\R^d)$ such that
  \begin{itemize}
  \item[(i)]
    $\ds\forall\lambda>0,\alpha\in\N_0^d:\ \sup_{\R^d}e^{\lambda\omega(x)}
    |D^\alpha u(x)|<+\infty$,
  \item[(ii)]
    $\ds\forall\lambda>0,\alpha\in\N_0^d:\ \sup_{\R^d}e^{\lambda\omega(\xi)}
    |D^\alpha \widehat{u}(\xi)|<+\infty$,
  \end{itemize}
  where $\N_0:=\N\cup\{0\}$ and $D^\alpha = (-i)^{\vert\alpha\vert} \partial^\alpha$.

  As usual, the corresponding dual space is denoted by $\Sch'_\omega(\R^d)$
  and is the set of all linear and continuous functionals
  $u:\,\Sch_\omega(\R^d)\to\C$. An element of $\Sch'_\omega(\R^d)$ is called an
  {\em $\omega$-tempered distribution}.
\end{Def}

    In \cite[Thm. 4.8]{BJO} we provided the space $\Sch_\omega(\R^d)$ with different equivalent systems of seminorms. For example, for $u\in\Sch_\omega(\R^d)$, the family of seminorms
    \beqs\label{seminorms}
    p_{\lambda,\mu}(u):=\sup_{\alpha,\beta\in\N_0^d}\sup_{x\in\R^d}|x^\beta D^\alpha u(x)|
    e^{-\lambda\varphi^*_\omega\left(\frac{|\alpha|}{\lambda}\right)-
      \mu\varphi^*_\omega\left(\frac{|\beta|}{\mu}\right)},
    \eeqs
    for $\lambda,\mu>0$. On the other hand, it is not difficult to see (using, for instance, \cite[Lemma 4.7(ii)]{BJO}) that the family of seminorms
\begin{equation}
\label{new-seminorms}
q_{\lambda,\mu}(u):=\sup_{\alpha\in\N_0^d}\sup_{x\in\R^{d}} |D^\alpha u(x)|
    e^{-\lambda\varphi^*_\omega\left(\frac{|\alpha|}{\lambda}\right)+
      \mu\omega(x)},\quad \lambda,\mu>0,
\end{equation}
defines another equivalent system of seminorms for $\Sch_{\omega}(\R^{d})$.

    We recall that $\Sch_\omega(\R^d)\subseteq\Sch(\R^d)$ and for their correspondent dual spaces we have the inclusion
    $\Sch'(\R^d)
    \subseteq\Sch'_\omega(\R^d)$.

Let us denote by $T_x$, $M_\xi$ and $\Pi(z)$, respectively, the
{\em translation}, the {\em modulation} and the {\em phase-space shift}
operators, defined by
\beqsn
T_xf(y)=f(y-x),\quad M_\xi f(y)=e^{i\langle y,\xi\rangle}f(y),\quad
\Pi(z)f(y)=M_\xi T_xf(y)=e^{i\langle y,\xi\rangle}f(y-x),
\eeqsn
for $x,y,\xi\in\R^d$ and $z=(x,\xi)$.

\begin{Def}
  \label{defSTFT}
  For a {\em window function} $\varphi\in\Sch_\omega(\R^d)\setminus\{0\}$, the
  {\em short-time Fourier transform} (briefly STFT) of $f\in\Sch'_\omega(\R^d)$
  is defined, for $z=(x,\xi)\in\R^{2d}$, by:
  \beqs
  \label{10}
  V_\varphi f(z):=&&\langle f,\Pi(z)\varphi\rangle\\
  \label{13}
  =&&\int_{\R^d}f(y)\overline{\varphi(y-x)}e^{-i\langle y,\xi\rangle}dy,
  \eeqs
  where the bracket $\langle \cdot,\cdot\rangle$ in \eqref{10}
and the integral in \eqref{13}
  denote the conjugate linear action of $\Sch'_\omega$ on $\Sch_\omega$,
  consistent with the inner product $\langle\cdot,\cdot\rangle_{L^2}$.
\end{Def}

By \cite[Lemma 1.1]{GZ}, for $f,\varphi,\psi\in\Sch_\omega(\R^d)$ we have the
following {\em inversion formula}:
\beqs
\label{12}
\langle\psi,\varphi\rangle f(y)=\frac{1}{(2\pi)^d}
\int_{\R^{2d}}V_\varphi f(z)(\Pi(z)\psi)(y)dz.
\eeqs
In particular, for $\psi=\varphi\in\Sch_\omega(\R^d)\setminus\{0\}$:
\beqs
\label{1}
f(y)=\frac{1}{(2\pi)^d\|\varphi\|^2_{L^2}}\int_{\R^{2d}}V_\varphi f(z)
(\Pi(z)\varphi)(y)dz.
\eeqs

We recall, from \cite{GZ}, the following results:

\begin{Th}
  \label{th24GZ}
  Let $\varphi\in\Sch_\omega(\R^d)\setminus\{0\}$ and $f\in\Sch'_\omega(\R^d)$.
  Then $V_\varphi f$ is continuous and there are constants $c,\lambda>0$
  such that
  \beqs
\label{Ad7}
  |V_\varphi f(z)|\leq c e^{\lambda\omega(z)}\qquad\forall z\in\R^{2d}.
  \eeqs
\end{Th}

\begin{Prop}
  \label{prop26GZ}
  Let $\varphi\in\Sch_\omega(\R^d)\setminus\{0\}$ and assume that
  $F:\ \R^{2d}\to\C$ is a measurable function that
  satisfies that for all $\lambda>0$ there is a constant
  $C_\lambda>0$ such that
  \beqsn
  |F(z)|\leq C_\lambda e^{-\lambda\omega(z)}\qquad\forall z\in\R^{2d}.
  \eeqsn
  Then
  \beqsn
  f(y):=\int_{\R^{2d}}F(z)(\Pi(z)\varphi)(y)dz
  \eeqsn
  defines a function $f\in\Sch_\omega(\R^d)$.
  \end{Prop}

\begin{Th}
  \label{th27GZ}
  Let $\varphi\in\Sch_\omega(\R^d)\setminus\{0\}$. Then, for
  $f\in\Sch'_\omega(\R^d)$, the following are equivalent:
  \begin{itemize}
  \item[(i)]
    $f\in\Sch_\omega(\R^d)$;
  \item[(ii)]
    for all $\lambda>0$ there exists $C_\lambda>0$ such that
    \beqsn
    |V_\varphi f(z)|\leq C_\lambda e^{-\lambda\omega(z)}\qquad
    \forall z\in\R^{2d};
    \eeqsn
  \item[(iii)]
    $V_\varphi f\in\Sch_\omega(\R^{2d})$.
  \end{itemize}
\end{Th}

The following lemma is well known for functions in $\Sch(\R^{d})$, and hence in $\Sch_\omega(\R^d).$ So we omit its proof.
\begin{Lemma}
\label{lemma11GZ}
For $f,\varphi\in\Sch_\omega(\R^d)$ we have that
\beqsn
\widehat{V_\varphi f}(\eta,y)=(2\pi)^de^{i\langle\eta,y\rangle}
f(-y)\overline{\widehat{\varphi}(\eta)}
\qquad\forall(\eta,y)\in\R^{2d}.
\eeqsn
\end{Lemma}

As a consequence, we can deduce the following result.
\begin{Prop}
\label{propcontS}
Let $\varphi\in\Sch_\omega(\R^d)\setminus\{0\}$. Then
\beqsn
V_\varphi:\ \Sch_\omega(\R^d)\longrightarrow\Sch_\omega(\R^{2d})
\eeqsn
is continuous.
\end{Prop}

\begin{proof}
Let us first remark that if $f\in\Sch_\omega(\R^d)$ then
$V_\varphi f\in\Sch_\omega(\R^{2d})$ by Theorem~\ref{th27GZ}.

Since $\Sch_{\omega}$ is a Fr\'echet space, to prove the continuity of
$V_\varphi$ we consider a sequence $\{f_n\}_{n\in\N}\subset\Sch_\omega(\R^d)$
such that
\beqs
\label{VS1}
f_n\longrightarrow f\in\Sch_\omega(\R^d)\qquad
\mbox{in}\ \Sch_\omega(\R^d)
\eeqs
and prove that $V_\varphi f_n\to V_\varphi f$ in $\Sch_\omega(\R^{2d})$.

Indeed, \eqref{VS1} implies that
\beqsn
e^{i\langle\eta,y\rangle}f_n(-y)\overline{\widehat{\varphi}(\eta)}
\longrightarrow e^{i\langle\eta,y\rangle}f(-y)\overline{\widehat{\varphi}(\eta)}
\qquad
\mbox{in}\ \Sch_\omega(\R^{2d})
\eeqsn
and hence, by Lemma~\ref{lemma11GZ},
\beqsn
\widehat{V_\varphi f_n}\rightarrow\widehat{V_\varphi f}
\qquad
\mbox{in}\ \Sch_\omega(\R^{2d}).
\eeqsn
Applying the inverse Fourier transform, which is continuous on $\Sch_\omega$,
we have that
\beqsn
V_\varphi f_n\rightarrow V_\varphi f
\qquad
\mbox{in}\ \Sch_\omega(\R^{2d}).
\eeqsn
and the proof is complete.
\end{proof}
The short-time Fourier transform also provides a new equivalent system of seminorms for $\Sch_{\omega}(\R^{d})$
\begin{Prop}
  \label{cor1126G}
  If $\varphi\in\Sch_\omega(\R^d)\setminus\{0\}$, then the collection of
  seminorms
  \beqsn
  \|V_\varphi f\|_{\omega,\lambda}:=\sup_{z\in\R^{2d}}|V_\varphi f(z)|e^{\lambda\omega(z)},
  \eeqsn
  for $\lambda>0$, forms an equivalent system of seminorms for
  $\Sch_\omega(\R^d)$.
\end{Prop}

\begin{proof}
  Set
  \beqsn
  \tilde{\Sch}_\omega(\R^d):=\{f\in\Sch(\R^d):\ \|V_\varphi f\|_{\omega,\lambda}
  <+\infty\ \forall\lambda>0\}.
  \eeqsn

  By Theorem~\ref{th27GZ} the sets $\tilde{\Sch}_\omega(\R^d)$ and
  $\Sch_\omega(\R^d)$ are equal. We have to prove that they have the same
  topology.

  By the inversion formula \eqref{1} we have that, for $z=(x,\xi)\in\R^{2d}$
  and $\lambda,\mu>0$,
  \beqs
  \nonumber
  \lefteqn{e^{-\lambda\varphi^*_\omega\left(\frac{|\alpha|}{\lambda}\right)}
   e^{-\mu\varphi^*_\omega\left(\frac{|\beta|}{\mu}\right)}|y^\beta D_y^\alpha f(y)|}\\
\nonumber
&&\leq
   C e^{-\lambda\varphi^*_\omega\left(\frac{|\alpha|}{\lambda}\right)}
   e^{-\mu\varphi^*_\omega\left(\frac{|\beta|}{\mu}\right)}\int_{\R^{2d}}
   |V_\varphi f(z)|\cdot|y^\beta D_y^\alpha
   (\Pi(z)\varphi)(y)|dz\\
   \nonumber
  && =C e^{-\lambda\varphi^*_\omega\left(\frac{|\alpha|}{\lambda}\right)}
   e^{-\mu\varphi^*_\omega\left(\frac{|\beta|}{\mu}\right)}\int_{\R^{2d}}|V_\varphi f(x,\xi)|\cdot
   |y^\beta D_y^\alpha e^{i\langle y,\xi\rangle}\varphi(y-x)|dxd\xi\\
   \label{2}
   &&\leq C\sum_{\gamma\leq\alpha}\binom{\alpha}{\gamma}2^{-|\alpha|}
   \int_{\R^{2d}}|V_\varphi f(x,\xi)|\cdot|y|^{|\beta|}
   e^{-\mu\varphi^*_\omega\left(\frac{|\beta|}{\mu}\right)}\\
\nonumber
&&\qquad \cdot|\xi|^{|\alpha-\gamma|}
   |D_y^\gamma\varphi(y-x)|e^{-\lambda\varphi^*_\omega\left(\frac{|\alpha|}{\lambda}\right)}
   2^{|\alpha|}dxd\xi
   \eeqs
   for some $C>0$.

   The following property is known and can be found, for instance, in Lemma 4.7(i) of \cite{BJO} (see also \cite{FGJ}):
   \beqs
   \label{3}
   |y|^{|\beta|}e^{-\mu\varphi^*_\omega\left(\frac{|\beta|}{\mu}\right)}\leq e^{\mu\omega(y)}, \mbox{ for all }y\in\R^{d}.
   \eeqs
   Moreover, since the weight function $\omega$ is subadittive and $\varphi^{*}$ is convex, we can use, for instance, Proposition~2.1(e) of \cite{BJ-Kotake} to obtain
   \beqs
   \label{4}
   2^{|\alpha|}e^{-\lambda\varphi^*_\omega\left(\frac{|\alpha|}{\lambda}\right)}\leq
   e^{3\lambda}e^{-3\lambda\varphi^*_\omega\left(\frac{|\alpha|}{3\lambda}\right)}.
   \eeqs

   Substituting \eqref{3} and \eqref{4} into \eqref{2}, by the subadditivity
of $\omega$ we have
   \beqs
   \nonumber
   e^{-\lambda\varphi^*_\omega\left(\frac{|\alpha|}{\lambda}\right)-\mu\varphi^*_\omega
     \left(\frac{|\beta|}{\mu}\right)}|y^\beta D_y^\alpha f(y)|
     \leq&& C_\lambda\sum_{\gamma\leq\alpha}\binom{\alpha}{\gamma}2^{-|\alpha|}
   \int_{\R^{2d}}|V_\varphi f(x,\xi)|e^{\mu\omega(x)}e^{\mu\omega(y-x)}\\
 \label{5}
&&\cdot|\xi|^{|\alpha-\gamma|}   |D_y^\gamma\varphi(y-x)|
   e^{-3\lambda\varphi^*_\omega\left(\frac{|\alpha|}{3\lambda}\right)}dxd\xi
   \eeqs
   for some $C_\lambda>0$.


   Since $\varphi\in\Sch_\omega(\R^{d})$, by \eqref{new-seminorms}, for every $\lambda,\mu>0$ there is a constant $C_{\lambda,\mu}>0$ such that for all $\gamma\in\N_{0}^{d}$ and $y\in\R^{d}$,
   \beqs
   |D_y^\gamma \varphi(y)|e^{\mu\omega(y)}\leq
   \label{6}
   C_{\lambda,\mu}
   e^{\lambda\varphi^*_\omega\left(\frac{|\gamma|}{\lambda}\right)}.
   \eeqs
      From \eqref{6} with $3\lambda$ instead of $\lambda$ and $y-x$ instead of $y$, we have that for every
   $\mu,\lambda>0$ there exists a constant $C_{\mu,\lambda}>0$
   such that
   \beqsn
   e^{-\lambda\varphi^*_\omega\left(\frac{|\alpha|}{\lambda}\right)-\mu\varphi^*_\omega
     \left(\frac{|\beta|}{\mu}\right)}|y^\beta D_y^\alpha f(y)|
   \leq&& C_{\mu,\lambda}\sum_{\gamma\leq\alpha}\binom{\alpha}{\gamma}2^{-|\alpha|}
   \\
   &&\cdot
   \int_{\R^{2d}}|V_\varphi f(x,\xi)|e^{\mu\omega(x)}|\xi|^{|\alpha-\gamma|}
   e^{3\lambda\varphi^*_\omega\left(\frac{|\gamma|}{3\lambda}\right)-3\lambda\varphi^*_\omega
     \left(\frac{|\alpha|}{3\lambda}\right)}dxd\xi.
   \eeqsn

   By \eqref{3} we have $|\xi|^{|\alpha-\gamma|}\le e^{3\lambda\omega(\xi)+3\lambda\varphi^*(\frac{|\alpha-\gamma|}{3\lambda})}$. So, from the convexity of $\varphi^{*}$, we obtain
   \beqs
   \nonumber
\lefteqn{e^{-\lambda\varphi^*_\omega\left(\frac{|\alpha|}{\lambda}\right)-\mu\varphi^*_\omega
     \left(\frac{|\beta|}{\mu}\right)}|y^\beta D_y^\alpha f(y)|}\\
   &&\leq C_{\mu,\lambda}\sum_{\gamma\leq\alpha}\binom{\alpha}{\gamma}2^{-|\alpha|}
   \int_{\R^{2d}}|V_\varphi f(x,\xi)|e^{\mu\omega(x)}e^{3\lambda\omega(\xi)}dxd\xi\\
   \nonumber
   &&\leq C_{\mu,\lambda}\int_{\R^{2d}}|V_\varphi f(z)|e^{(\mu+3\lambda+1)\omega(z)}
   e^{-\omega(z)}dz\\
   \label{7}
   &&\leq C'_{\mu,\lambda}\|V_\varphi f\|_{\omega,\mu+3\lambda+1},
   \eeqs
   for $C'_{\mu,\lambda}:=C_{\mu,\lambda} \int_{\R^{2d}}
   e^{-\omega(z)}dz$, which is finite by condition $(\gamma)$ of Definition~\ref{defweight}.

   It is easy to see that $\tilde{\Sch}_\omega(\R^d)$ is a Fr\'echet space. Indeed, the estimate \eqref{7} implies that the identity operator
   $I:\ \tilde{\Sch}_\omega(\R^d)\to\Sch_\omega(\R^d)$ is continuous. Hence, any Cauchy sequence $\{f_n\}_{n\in\N}$  in
$\tilde{\Sch}_\omega(\R^d)$ is a Cauchy sequence in
$\Sch_\omega(\R^d)$. So, it converges in $\Sch_{\omega}(\R^{d})$ to some $f$ (because $\Sch_{\omega}(\R^{d})$ is complete). From Proposition~\ref{propcontS}, $\{V_\varphi f_n\}_{n\in\N}$ converges to $V_\varphi f$ in $\Sch_\omega(\R^{2d})$. Therefore, $\{f_n\}_{n\in\N}$ converges to  $f$ in $\tilde{\Sch}_\omega(\R^d).$

We can
apply the open mapping theorem to conclude that $I$
   is an isomorphism and hence the two topologies on $\Sch_\omega(\R^d)$
   coincide.
   \end{proof}

Now, we can prove the following

\begin{Prop}
  \label{cor1127G}
  Assume that $\psi,\gamma\in\Sch_\omega(\R^d)\setminus\{0\}$ with
  $\langle\psi,\gamma\rangle\neq0$. Then the following assertions hold:
  \begin{itemize}
  \item[(a)]
    If $F:\ \R^{2d}\to\C$ is a measurable function that satisfies, for some $c,\lambda>0$,
    \beqs
    \label{9}
    |F(z)|\leq ce^{\lambda\omega(z)}\qquad\forall z\in\R^{2d},
    \eeqs
    then
    \beqsn
    \Sch_\omega(\R^d)\ni\varphi\mapsto\langle f,\varphi\rangle:=\int_{\R^{2d}}
    F(z)\langle\Pi(z)\gamma,\varphi\rangle dz
    \eeqsn
    define an $\omega$-tempered distribution $f\in\Sch'_\omega(\R^d)$.
  \item[(b)]
    In particular, if $F=V_\psi f$ for some $f\in\Sch'_\omega(\R^d)$, then the
    following {\em inversion formula} holds:
    \beqs
    \label{8}
    f=\frac{1}{(2\pi)^d\langle\gamma,\psi\rangle}\int_{\R^{2d}}
    V_\psi f(z)\Pi(z)\gamma dz.
    \eeqs
  \end{itemize}
\end{Prop}

\begin{proof}
  From  \eqref{9} we have, for all
  $\varphi\in\Sch_\omega(\R^d)$,
  \beqs
  \nonumber
  |\langle f,\varphi\rangle|\leq&&\int_{\R^{2d}}|F(z)|\cdot|V_\gamma \varphi(z)|dz\\
  \nonumber
  \leq &&c\int_{\R^{2d}}e^{\lambda\omega(z)+\omega(z)}|V_\gamma\varphi(z)|e^{-\omega(z)}dz\\
  \label{11}
  \leq&&c' \|V_\gamma\varphi\|_{\omega,\lambda+1}
  \eeqs
  for some $c'>0$.

  From Proposition~\ref{cor1126G} the inequality \eqref{11} implies that
  $f$ defines a continuous linear functional on $\Sch_\omega(\R^d)$, i.e.
  $f\in\Sch'_\omega(\R^d)$.
  This proves $(a)$.

  In particular, if $F=V_\psi f$ for some $f\in\Sch'_\omega(\R^d)$ then
  $F$ satisfies \eqref{9} for Theorem~\ref{th24GZ} and hence \eqref{8}
  defines an $\omega$-tempered distribution $\tilde{f}\in\Sch'_\omega(\R^d)$
 given by
  \beqsn
  \langle\tilde{f},\varphi\rangle=&&\frac{1}{(2\pi)^d\langle\gamma,\psi\rangle}
  \int_{\R^{2d}}V_\psi f(z)\langle\Pi(z)\gamma,\varphi\rangle dz
  \qquad\forall\varphi\in\Sch_\omega(\R^d).
  \eeqsn

  However, from \eqref{12} we have that
  \beqsn
  \varphi=&&\frac{1}{(2\pi)^d\langle\psi,\gamma\rangle}\int_{\R^{2d}}V_\gamma\varphi(z)\Pi(z)\psi dz
  \eeqsn
  and then (see also \cite[pg 43]{G} for vector valued integrals)
  \beqsn
  \langle f,\varphi\rangle=&&\frac{1}{(2\pi)^d\overline{\langle\psi,\gamma\rangle}}\int_{\R^{2d}}\overline{V_\gamma\varphi(z)}\langle f,\Pi(z)\psi \rangle dz \\
  =&&\frac{1}{(2\pi)^d\langle\gamma,\psi\rangle}\int_{\R^{2d}}
\langle\Pi(z)\gamma,\varphi\rangle
  V_\psi f(z)dz\\
  =&&\langle\tilde{f},\varphi\rangle,\qquad\varphi\in\Sch_\omega(\R^d).
  \eeqsn
Therefore $f=\tilde f$ and $(b)$ is proved.
  \end{proof}

Let us now recall the definition of the adjoint operator of $V_\varphi$.
We consider, for $\varphi\in L^2(\R^d)$, the operator
\beqsn
A_\varphi:\ L^2(\R^{2d})\longrightarrow L^2(\R^d)
\eeqsn
defined by
\beqsn
A_\varphi F=\int_{\R^{2d}}F(z)\Pi(z)\varphi\, dz.
\eeqsn

This is the adjoint operator of $V_\varphi: L^2(\R^d)\to L^2(\R^d)$ since,
for all $F\in L^2(\R^{2d})$ and $h\in L^2(\R^d)$,
\beqsn
\langle A_\varphi F,h\rangle=\int_{\R^{2d}}F(z)\langle\Pi(z)\varphi,h\rangle
dz
=\langle F,V_\varphi h\rangle=\langle V_\varphi^* F,h\rangle.
\eeqsn

In particular, for $\varphi\in\Sch_\omega(\R^d)$ and $F\in\Sch_\omega(\R^{2d})$
we can define the adjoint operator $V_\varphi^* F=A_\varphi F$. We observe that
$V_\varphi^* F\in\Sch_\omega(\R^d)$. In fact, if
$G(x,\xi,t):=F(x,\xi)\varphi(t-x)\in\Sch_\omega(\R^{3d})$, we can write
$A_\varphi F$ as a partial Fourier transform:
\beqs
\label{Ad5}
A_\varphi F(t)=\int_{\R^{2d}}F(x,\xi)\varphi(t-x)e^{i\langle t,\xi\rangle}dxd\xi
=\left.\left(\F_{(x,\xi)}G\right)(x',\xi',t)\right|_{(x',\xi',t)=(0,-t,t)}.
\eeqs

Then
\beqs
\label{Eq1}
V_\varphi^*:\ \Sch_\omega(\R^{2d})\longrightarrow\Sch_\omega(\R^d)
\eeqs
continuously.

Moreover, the inversion formula \eqref{12} gives, for
$\varphi,\psi,f\in \Sch_\omega(\R^d)$ with $\langle\varphi,\psi\rangle\neq0$,
\beqsn
\frac{1}{\langle\varphi,\psi\rangle}V_\varphi^*V_\psi f=
\frac{1}{\langle\varphi,\psi\rangle}\int_{\R^{2d}}V_\psi f(z)\Pi(z)\varphi dz
=(2\pi)^d f,
\eeqsn
i.e.
\beqs
\label{Ad9}
\frac{1}{(2\pi)^d\langle\varphi,\psi\rangle}V_\varphi^*V_\psi
=I_{\Sch_\omega(\R^d)}.
\eeqs

More in general, if $\varphi\in\Sch_\omega(\R^d)\setminus\{0\}$ and $F$ is a
measurable function on $\R^{2d}$, we define the adjoint operator
\beqs
\label{Ad1}
V_\varphi^* F=\int_{\R^{2d}}F(z)\Pi(z)\varphi dz,
\eeqs
where the integral is interpreted, if necessary, in a weak sense, i.e.
\beqsn
\langle V_\varphi^*F,g\rangle=\int_{\R^{2d}}F(z)\langle\Pi(z)\varphi,g\rangle
dz
=\int_{\R^{2d}}F(z)\overline{V_\varphi g}(z)dz
=\langle F, V_\varphi g\rangle
\eeqsn
for $g\in\Sch_\omega(\R^d)$.

In particular, if $\varphi,\psi\in\Sch_\omega(\R^d)\setminus\{0\}$ with
$\langle\varphi,\psi\rangle\neq0$, by Theorem~\ref{th24GZ} and
Proposition~\ref{cor1127G} we can define the adjoint operator \eqref{Ad1}
for $F=V_\psi f$ with $f\in\Sch'_\omega(\R^d)$ and obtain that, for all
$g\in\Sch_\omega(\R^d)$,
\beqs
\label{Ad21}
\langle V_\varphi^*V_\psi f,g\rangle
=\int_{\R^{2d}}V_\psi f(z)\langle\Pi(z)\varphi,g\rangle dz
=(2\pi)^d\langle\varphi,\psi\rangle\langle f,g\rangle,
\eeqs
i.e.
\beqs
\label{Ad2}
\frac{1}{(2\pi)^d\langle\varphi,\psi\rangle}V_\varphi^*V_\psi
=I_{\Sch'_\omega(\R^d)}.
\eeqs

We can now prove the following proposition in a standard way.

  \begin{Prop}
  \label{lemma1133G}
  Let $\varphi,\psi,\gamma\in\Sch_\omega(\R^d)$ with $\langle\gamma,\psi\rangle\neq0$
  and let $f\in\Sch'_\omega(\R^d)$. Then
  \beqsn
  |V_\varphi f(z)|\leq\frac{1}{(2\pi)^d|\langle\gamma,\psi\rangle|}
  (|V_\psi f|\ast|V_\varphi\gamma|)(z),\qquad z=(x,\xi)\in\R^{2d}.
  \eeqsn
    \end{Prop}

\section{The $\omega$-Gabor wave front set}
\label{sec3}
   In this section, we consider a global wave front set defined in terms of rapid decay
of the
   STFT in conical sets. After that, for a Gabor frame we define the {\em Gabor wave front set}, where conical sets are
   intersected with a lattice. We prove that these wave front sets coincide.

\begin{Def}
\label{def31RW}
Let $u\in\Sch'_\omega(\R^d)$ and $\varphi\in\Sch_\omega(\R^d)\setminus\{0\}$. We say that
$z_0=(x_0,\xi_0)\in\R^{2d}\setminus\{0\}$ is not in the $\omega$-wave front set $\WF'_\omega(u)$ of
$u$ if there exists an open conic set $\Gamma\subseteq\R^{2d}\setminus\{0\}$
containing $z_0$ and such that
\beqs
\label{16}
\sup_{z\in\Gamma}e^{\lambda\omega(z)}| V_\varphi u(z)|<+\infty,
\qquad\lambda>0.
\eeqs
\end{Def}

We observe that $\WF'_\omega(u)$ is a closed conic subset of
$\R^{2d}\setminus\{0\}$. Moreover, it does not depend on the choice of the window
function $\varphi$, as the following proposition shows.
\begin{Prop}
\label{cor33RW}
Let $u\in\Sch'_\omega(\R^d)$, $\varphi\in\Sch_\omega(\R^d)\setminus\{0\}$ and
$z_0\in\R^{2d}\setminus\{0\}$. Assume that there exists an open conic set
$\Gamma\subseteq\R^{2d}\setminus\{0\}$ containing $z_0$ such that
\eqref{16} is satisfied.
Then, for any $\psi\in\Sch_\omega(\R^d)\setminus\{0\}$ and for any open conic set
$\Gamma'\subseteq\R^{2d}\setminus\{0\}$ containing $z_0$ and such that
$\overline{\Gamma'\cap S_{2d-1}}\subseteq\Gamma$, where $S_{2d-1}$ is the unit
sphere in $\R^{2d}$, we have
\beqs
\label{22}
\sup_{z\in\Gamma'}e^{\lambda\omega(z)}|V_\psi u(z)|<+\infty,
\qquad\lambda>0.
\eeqs
\end{Prop}

\begin{proof}
From Proposition~\ref{lemma1133G} we have that
\beqs
\label{21}
|V_\psi u(z)|\leq (2\pi)^{-d}\|\varphi\|^{-2}_{L^2}(|V_\varphi u|\ast|V_\psi\varphi|)(z)
\qquad\forall z\in\R^{2d}.
\eeqs

Moreover, since $\varphi\in\Sch_\omega(\R^{d})$, from Theorem~\ref{th27GZ}
we have that for every $\mu>0$ there exists $C_\mu>0$ such that
\beqs
\label{17}
e^{\mu\omega(z)}|V_\psi\varphi(z)|\leq C_\mu\qquad
\forall z\in\R^{2d}.
\eeqs

Then
\beqs
\nonumber
(|V_\varphi u|\ast|V_\psi\varphi|)(z)&&=\int_{\R^{2d}}|V_\varphi u(z-z')|\cdot
|V_\psi\varphi(z')|dz'\\
\nonumber
&&\leq\int_{\langle z'\rangle\leq\varepsilon\langle z\rangle}|V_\varphi u(z-z')|
\cdot|V_\psi\varphi(z')|dz'
+\int_{\langle z'\rangle>\varepsilon\langle z\rangle}|V_\varphi u(z-z')|
\cdot|V_\psi\varphi(z')|dz'\\
\label{19}
&&=:I_1+I_2.
\eeqs

Let us choose $\varepsilon>0$ sufficiently small so that
\beqsn
z\in\Gamma',\ |z|\geq1,\ \langle z'\rangle\leq\varepsilon\langle z\rangle\quad
\Rightarrow\quad z-z'\in\Gamma,
\eeqsn
and hence, from \eqref{16}, the subadditivity of $\omega$ and \eqref{17}:
\beqs
\nonumber
I_1\leq&&C_\lambda\int_{\langle z'\rangle\leq\varepsilon\langle z\rangle}
e^{-\lambda\omega(z-z')}|V_\psi\varphi(z')|dz'\\
\nonumber
\leq&&C_\lambda e^{-\lambda\omega(z)}\int_{\R^{2d}}e^{(\lambda+1)\omega(z')}
|V_\psi\varphi(z')|e^{-\omega(z')}dz'\\
\label{18}
\leq&&C'_\lambda e^{-\lambda\omega(z)},\qquad\lambda>0,\  z\in\Gamma',\ |z|\geq1.
\eeqs
On the other hand, from Theorem~\ref{th24GZ} and \eqref{17}:
\beqs
\nonumber
I_2\leq&&c\int_{\langle z'\rangle>\varepsilon\langle z\rangle}
e^{\lambda\omega(z-z')}|V_\psi\varphi(z')|dz'\\
\nonumber
\leq&&ce^{\lambda\omega(z)}\int_{\langle z'\rangle>\varepsilon\langle z\rangle}
e^{(\lambda+1-\mu)\omega(z')}|V_\psi\varphi(z')|e^{\mu\omega(z')}e^{-\omega(z')}dz'\\
\label{19bis}
\leq&&c' e^{\lambda\omega(z)}e^{A_\varepsilon(\lambda+1-\mu)\omega(z)}C_\mu
\eeqs
for some $c'>0$,
if  $\mu>\lambda+1$, since
\beqsn
-\omega(z')\leq -A_\varepsilon(1+\omega(z)),
\qquad\mbox{if}\ \langle z'\rangle>\varepsilon\langle z\rangle,
\eeqsn
for some constant $A_\varepsilon>0$ which depends on the already fixed $\varepsilon>0$.

The arbitrariness of $\mu>\lambda+1$ in \eqref{19bis} implies that for every $\lambda'>0$
there exists a constant $C_{\lambda'}>0$ such that
\beqs
\label{20}
I_2\leq C_{\lambda'}e^{-\lambda'\omega(z)},\quad z\in\R^{2d}.
\eeqs
This gives the conclusion.
\end{proof}

Given $\alpha, \beta>0$, consider the \emph{lattice} $\Lambda=\alpha \Z^{d}\times \beta \Z^{d}\subset\R^{2d}.$ For a window $\varphi\in L^{2}(\R^{d})\setminus\{0\}$ the collection $\{\Pi(\sigma)\varphi\}_{\sigma\in\Lambda}$ is called a \emph{Gabor frame} for $L^{2}(\R^{d})$ provided there exists constants $A,B>0$ such that
$$
A\|f\|_{L^{2}}^{2}\le \sum_{\sigma\in\Lambda} |\langle f,\Pi(\sigma)\varphi\rangle|^{2}\le B \|f\|_{L^{2}}^{2},\quad f\in L^{2}(\R^{d})
$$
(see \cite{G} for the analysis of the conditions on $\alpha$ and $\beta$ for which $\{\Pi(\sigma)\varphi\}_{\sigma\in\Lambda}$ is a Gabor frame). Now, we define the Gabor $\omega$-wave front set.
\begin{Def}
\label{def34RW}
Let $\varphi\in\Sch_\omega(\R^d)\setminus\{0\}$ and $\Lambda=\alpha_0\Z^d\times
\beta_0\Z^d\subseteq\R^{2d}$ a lattice with $\alpha_0,\beta_0>0$ sufficiently small so that
$\{\Pi(\sigma)\varphi\}_{\sigma\in\Lambda}$ is a Gabor frame for $L^2(\R^d)$.
If $u\in\Sch'_\omega(\R^d)$, we say that $z_0\in\R^{2d}\setminus\{0\}$ is not in the {\em Gabor $\omega$-wave front set} $\WF^G_\omega(u)$ of $u$ if there exists an open conic
set $\Gamma\subset\R^{2d}\setminus\{0\}$ containing $z_0$ such that
\beqs
\label{23}
\sup_{\sigma\in\Lambda\cap\Gamma}e^{\lambda\omega(\sigma)}|V_\varphi u(\sigma)|
<+\infty\qquad
\forall\lambda>0.
\eeqs
\end{Def}

Our aim is to  prove that $\WF'_\omega(u)=\WF^G_\omega(u)$. We follow the lines of \cite{RW} and  we need some properties of modulation spaces that are already true in $\Sch(\R^{d})$ (see \cite{G}), adapted to our setting.

We consider, for $\lambda\in\R\setminus\{0\}$,
\beqs
\label{21bis}
m_\lambda(z)=e^{\lambda\omega(z)},\qquad
v_\lambda(z)=e^{\vert\lambda\vert \omega(z)},\qquad z\in\R^n.
\eeqs
The weights $m_\lambda(z)$ are $v_\lambda$-moderate, in the sense that
\beqsn
m_\lambda(z_1+z_2)\leq v_\lambda(z_1)m_\lambda(z_2),
\eeqsn
for every $\lambda\neq 0$ and $z_1,z_2\in \R^n$. This is immediate from the subadditivity of $\omega$.

We denote, following \cite{G}, the weighted $L^{p,q}$ spaces by
\beqsn
L^{p,q}_{m_\lambda}(\R^{2d}):=\Big\{&&\!\!\!F\ \mbox{measurable on $\R^{2d}$ such that}\\
&&\|F\|_{L^{p,q}_{m_\lambda}}:=\Big(\int_{\R^{d}}\Big(\int_{\R^{d}}|F(x,\xi)|^pm_\lambda (x,\xi)^pdx\Big)^{q/p}
d\xi\Big)^{1/q}<+\infty\Big\},
\eeqsn
for $1\leq p,q<+\infty$, and
\beqsn
L^{\infty,q}_{m_\lambda}(\R^{2d}):=\Big\{&&\!\!\!F\ \mbox{measurable on $\R^{2d}$ such that}\\
&&\|F\|_{L^{\infty,q}_{m_\lambda}}:=\Big(\int_{\R^{d}}\big(\esssup_{x\in\R^d}|F(x,\xi)|
m_\lambda(x,\xi)\big)^{q}
d\xi\Big)^{1/q}<+\infty\Big\},\\
L^{p,\infty}_{m_\lambda}(\R^{2d}):=\Big\{&&\!\!\!F\ \mbox{measurable on $\R^{2d}$ such that}\\
&&\|F\|_{L^{p,\infty}_{m_\lambda}}:=\esssup_{\xi\in\R^d}
\Big(\int_{\R^{d}}|F(x,\xi)|^pm_\lambda(x,\xi)^p dx\Big)^{1/p}
<+\infty\Big\},
\eeqsn
for $1\leq p,q\leq+\infty$ with $p=+\infty$ or $q=+\infty$ respectively.

By \cite[Lemma 11.1.2]{G} these are Banach spaces for all $1\leq p,q\leq+\infty$. Moreover, for $F\in L^{p,q}_{m_\lambda}(\R^{2d})$ and $H\in L^{p',q'}_{1/{m_\lambda}}(\R^{2d})$, where
$p'$ and $q'$ are the conjugate exponents of $p$ and $q$ respectively
(i.e. $\frac1p+\frac{1}{p'}=1$ if $1<p<+\infty$, $p'=+\infty$ if $p=1$, $p'=1$ if
$p=+\infty$, and the same for $q$), then $F\cdot H\in L^1(\R^{2d})$ and
\beqs
\label{Holder}
\left|\int_{\R^{2d}}F(z)\overline{H(z)}dz\right|\leq\|F\|_{L^{p,q}_{m_\lambda}}
\|H\|_{L^{p',q'}_{1/m_\lambda}}.
\eeqs
If $1\leq p,q<+\infty$, the dual of $L^{p,q}_{m_\lambda}(\R^{2d})$ is given by
$L^{p',q'}_{1/m_\lambda}(\R^{2d})$.

From \cite[Proposition 11.1.3]{G} we have the following Young inequality for weighted $L^{p,q}$ spaces. For $F\in L^{p,q}_{m_\lambda}$ and $G\in L^1_{v_\lambda}$,
\beqs
\label{Young}
\Vert F*G\Vert_{L^{p,q}_{m_\lambda}}\leq C\Vert F\Vert_{L^{p,q}_{m_\lambda}} \Vert G\Vert_{L^1_{v_\lambda}},
\eeqs
for some $C>0$.

\begin{Rem}\label{SLpq}
\begin{em}It is easy to see that for every $\lambda\in\R\setminus\{0\}$ and $1\leq p,q\leq +\infty$ we have
\beqsn
\Sch_\omega(\R^{2d})\subset L^{p,q}_{m_\lambda}(\R^{2d}).
\eeqsn
\end{em}
\end{Rem}

\begin{Def}
\label{defmod}
Let $\varphi\in\Sch_\omega(\R^d)\setminus\{0\}$, and $m_\lambda(z)$ as in \eqref{21bis} for some
$\lambda\neq0$.
For $1\leq p,q\leq+\infty$, the {\em modulation space} $\boldsymbol{M}^{p,q}_{m_\lambda}(\R^d)$
is defined by
\beqsn
\boldsymbol{M}^{p,q}_{m_\lambda}(\R^d):=\{f\in\Sch'_\omega(\R^d):\ V_\varphi f\in
L^{p,q}_{m_\lambda}(\R^{2d})\},
\eeqsn
with norm $\Vert f\Vert_{\boldsymbol{M}^{p,q}_{m_\lambda}}=\Vert V_\varphi f\Vert_{L^{p,q}_{m_\lambda}}$. We denote then $\boldsymbol{M}^{p}_{m_\lambda}(\R^d):=\boldsymbol{M}^{p,p}_{m_\lambda}(\R^d)$.
\end{Def}

Observe that Definition~\ref{defmod} is similar to the definition of modulation spaces in \cite{G}; the difference is that here $\boldsymbol{M}^{p,q}_{m_\lambda}(\R^d)$ is a subset of $\Sch'_\omega(\R^d)$, and we take a window $\varphi\in\Sch_\omega(\R^d)$, while in \cite{G} the modulation space $M^{p,q}_m(\R^d)$ is a subset of $\Sch'(\R^d)$  and the window belongs to $\Sch(\R^d)$ (or a subset of $(M^1_v)^*$ for a suitable weight $v$, in a suitable space of `special' windows $\Sch_{\mathcal{C}}(\R^d)$). Moreover, here we always need weights of exponential type. We refer to \cite{T-1,T-2} for modulation spaces in the setting of Gelfand-Shilov spaces, among other type of spaces of ultradifferentiable functions and ultradistributions.

The definition of $\boldsymbol{M}^{p,q}_{m_\lambda}$ is independent of the window $\varphi$, in the sense that different (non-zero) windows in $\Sch_\omega(\R^d)$ give equivalent norms. Indeed for $\varphi,\psi\in\Sch_\omega(\R^d)$, $\varphi,\psi\neq 0$, we have from Proposition~\ref{lemma1133G}, applied with $\gamma=\psi$, that
\beqs
\label{prop1132(c)G}
\Vert V_\varphi f\Vert_{L^{p,q}_{m_\lambda}}\leq \frac{1}{(2\pi)^d \Vert \psi\Vert_{L^2}^2} \Vert \vert V_\psi f\vert * \vert V_\varphi\psi\vert\Vert_{L^{p,q}_{m_\lambda}}\leq C\Vert V_\psi f\Vert_{L^{p,q}_{m_\lambda}},
\eeqs
where $C=\frac{\Vert V_\varphi\psi\Vert_{L^1_{v_\lambda}}}{(2\pi)^d \Vert \psi
\Vert_{L^2}^2}$, as we can deduce from Young inequality \eqref{Young} (observe that $C$ is finite by Proposition~\ref{propcontS} and Remark~\ref{SLpq}). Then, by interchanging the roles of $\varphi$ and $\psi$ we have that $V_\varphi f\in L^{p,q}_{m_\lambda}$ if and only if $V_\psi f\in L^{p,q}_{m_\lambda}$, and the corresponding modulation space norms of $f$ with respect to the two windows are equivalent.

\begin{Rem}
\begin{em}
\label{form25RW}
From Theorems~\ref{th27GZ} and \ref{th24GZ} and Proposition~\ref{cor1127G}
we have that
\beqsn
\Sch_\omega(\R^d)=\bigcap_{\lambda>0}\boldsymbol{M}^\infty_{m_\lambda}(\R^d);
\qquad\Sch'_\omega(\R^d)=\bigcup_{\lambda<0}\boldsymbol{M}^\infty_{m_\lambda}(\R^d).
\eeqsn
\end{em}
\end{Rem}

The inversion formula of Proposition~\ref{cor1127G} holds also in modulation spaces, as stated in the following result.

\begin{Prop}
\label{prop1132G}
Let $\gamma\in\Sch_\omega(\R^d)$ be a not identically zero window, and
consider, for a measurable function
$F$ on $\R^{2d}$, the adjoint $V_\gamma^*F$ defined as in \eqref{Ad1}.
Then:
\begin{itemize}
\item[(i)]
The operator $V^*_\gamma$ acts continuously as
\beqsn
V^*_\gamma:L^{p,q}_{m_\lambda}(\R^{2d})\to \boldsymbol{M}^{p,q}_{m_\lambda}(\R^d),
\eeqsn
and there exists $C>0$ such that
\beqsn
\Vert V^*_\gamma F\Vert_{\boldsymbol{M}^{p,q}_{m_\lambda}}\leq C\Vert V_\varphi \gamma\Vert_{L^1_{v_\lambda}} \Vert F\Vert_{L^{p,q}_{m_\lambda}},
\eeqsn
where $\varphi$ is the window in the corresponding $\boldsymbol{M}^{p,q}_{m_\lambda}$ norm.
\item[(ii)]
In the particular case when $F=V_g f$, for $g\in\Sch_\omega(\R^d)$, and $f\in \boldsymbol{M}^{p,q}_{m_\lambda}$, if $\langle\gamma,g\rangle\neq 0$ the following {\em inversion formula} holds:
\beqsn
f=\frac{1}{(2\pi)^d\langle\gamma,g\rangle} \int_{\R^{2d}} V_g f(z) \Pi(z)\gamma\,dz.
\eeqsn
\end{itemize}
\end{Prop}

\begin{proof}

{\rm (i)} We start by proving that $V^*_\gamma F$ is an element of $\Sch^\prime_\omega(\R^d)$. For $\psi\in\Sch_\omega(\R^d)$ we have from \eqref{Holder},
\beqsn
\vert\langle V^*_\gamma F,\psi\rangle\vert &=& \vert\langle F,V_\gamma \psi\rangle\vert \leq \Vert F\Vert_{L^{p,q}_{m_\lambda}} \Vert V_\gamma\psi\Vert_{L^{p^\prime,q^\prime}_{1/m_\lambda}} \\
&\leq& \Vert F\Vert_{L^{p,q}_{m_\lambda}} \Vert e^{\mu\omega(z)} V_\gamma\psi\Vert_\infty \Vert e^{-\mu\omega(z)}\Vert_{L^{p^\prime,q^\prime}_{1/m_\lambda}};
\eeqsn
this expression is finite for $\mu>0$ sufficiently large, as we can deduce from Theorem~\ref{th27GZ}(ii) and Definition~\ref{defweight}. Then from Proposition~\ref{cor1126G} we have that $V^*_\gamma F$ is a well defined element of $\Sch^\prime_\omega(\R^{2d})$. From Theorem~\ref{th24GZ} we have that $V_\varphi V^*_\gamma F$ is a continuous function; it is explicitly given by
\beqsn
V_\varphi V^*_\gamma F(z) &=& \langle V^*_\gamma F,\Pi(z)\varphi\rangle = \int_{\R^{2d}}F(y,\eta) \overline{V_\gamma(\Pi(z)\varphi)(y,\eta)}\,dy\,d\eta.
\eeqsn
Writing $z=(x,\xi)$ we have
\beqsn
\vert V_\varphi V^*_\gamma F(x,\xi)\vert &=& \left\vert \int_{\R^{2d}}F(y,\eta) V_\varphi\gamma(x-y,\xi-\eta) e^{- i \langle y,\xi-\eta\rangle}\,dy\,d\eta\right\vert \\
&\leq& (\vert F\vert * \vert V_\varphi\gamma\vert)(x,\xi).
\eeqsn
Then, from Young inequality \eqref{Young} we obtain
\beqs
\label{68}
\Vert V^*_\gamma F\Vert_{\boldsymbol{M}^{p,q}_{m_\lambda}} = \Vert V_\varphi V^*_\gamma F\Vert_{L^{p,q}_{m_\lambda}} \leq C \Vert F\Vert_{L^{p,q}_{m_\lambda}} \Vert V_\varphi\gamma\Vert_{L^1_{v_\lambda}},
\eeqs
and this expression is finite since $V_\varphi\gamma\in\Sch_\omega(\R^{2d})\subset L^1_{v_\lambda}(\R^{2d})$ for every $\lambda\in\R$ from Remark~\ref{SLpq}.

{\rm (ii)} We first observe that, by \eqref{prop1132(c)G}, $V_g f\in L^{p,q}_{m_\lambda}$. Then, from point (i), $\tilde{f}=\frac{1}{(2\pi)^d \langle\gamma,g\rangle} V^*_\gamma V_g f\in \boldsymbol{M}^{p,q}_{m_\lambda}$. Since $\boldsymbol{M}^{p,q}_{m_\lambda}\subset\Sch^\prime_\omega$,  we have that $\tilde{f}=f$ by
\eqref{Ad2}.
\end{proof}

\begin{Th}\label{Th1136G}
Let $1\leq p,q<\infty$. We have
\beqsn
(\boldsymbol{M}^{p,q}_{m_\lambda})^*=\boldsymbol{M}^{p^\prime,q^\prime}_{1/m_\lambda},
\eeqsn
and the duality is given by
\beqsn
\langle f,h\rangle = \int_{\R^{2d}} V_\varphi f(z) \overline{V_\varphi h(z)}\,dz
\eeqsn
for $f\in\boldsymbol{M}^{p,q}_{m_\lambda}$ and $h\in \boldsymbol{M}^{p^\prime,q^\prime}_{1/m_\lambda}$.
\end{Th}

\begin{proof}
The proof of this result relies on the duality of weighted $L^{p,q}$ spaces, and it is the same as in Theorem 11.3.6 of \cite{G}.
\end{proof}

\begin{Prop}\label{Prop1134G}
For $1\leq p,q<\infty$ we have that $\Sch_\omega(\R^d)$ is a dense subspace of $\boldsymbol{M}^{p,q}_{m_\lambda}$.
\end{Prop}

\begin{proof}
We first observe that, from property $(\gamma)$ of the weight function $\omega$ (see Definition \ref{defweight}) we have that, for $\mu>0$, $e^{-\mu\omega(z)}\in L^{p,q}_{m_\lambda}$. Hence, for every $f\in\Sch_\omega(\R^d)$ we obtain
\beqsn
\Vert f\Vert_{\boldsymbol{M}^{p,q}_{m_\lambda}} = \Vert V_\varphi f\Vert_{L^{p,q}_{m_\lambda}}\leq \Vert V_\varphi f(z) e^{\mu\omega(z)}\Vert_\infty \Vert e^{-\mu\omega(z)}\Vert_{L^{p,q}_{m_\lambda}}.
\eeqsn
From Proposition \ref{cor1126G} we have
\beqsn
\Sch_\omega(\R^d)\subset \boldsymbol{M}^{p,q}_{m_\lambda},
\eeqsn
with continuous inclusion. It remains to prove the density. We denote by $K_n:=\{ z\in\R^{2d} : \vert z\vert\leq n\}$, and we fix $\varphi\in\Sch_\omega$ with $\Vert \varphi\Vert_{L^2}^2=(2\pi)^{-d}$. Consider $f\in\boldsymbol{M}^{p,q}_{m_\lambda}$ and define
\beqsn
F_n = V_\varphi f\cdot \chi_{K_n}\quad\text{and}\quad f_n = V^*_\varphi F_n.
\eeqsn
From Proposition \ref{prop26GZ} we have that $f_n\in\Sch_\omega(\R^d)$. Moreover,
using \eqref{Ad2} and Proposition \ref{prop1132G} we obtain
\beqsn
\Vert f_n-f\Vert_{\boldsymbol{M}^{p,q}_{m_\lambda}} &=& \Vert V^*_\varphi F_n - V^*_\varphi V_\varphi f\Vert_{\boldsymbol{M}^{p,q}_{m_\lambda}} \\
&\leq& C \Vert F_n - V_\varphi f\Vert_{L^{p,q}_{m_\lambda}} \\
&=& C \Vert V_\varphi f\Vert_{L^{p,q}_{m_\lambda}(\R^{2d}\setminus K_n)}.
\eeqsn
So, $\Vert f_n-f\Vert_{\boldsymbol{M}^{p,q}_{m_\lambda}}$ tends to $0$ for $n\to \infty$, which finishes the proof.
\end{proof}

We recall now from \cite{G} some basic facts about {\em amalgam spaces}.
\begin{Def}
\label{def1113G}
We indicate with $\ell^{p,q}_{m_\lambda}(\Z^{2d})$ the space of all sequences $(a_{kn})_{k,n\in\Z^d}$, with $a_{kn}\in\C$ for every $k,n\in\Z^d$, such that the following norm is finite
\beqsn
\Vert a\Vert_{\ell^{p,q}_{m_\lambda}} = \biggl(\sum_{n\in\Z^d}\biggl(\sum_{k\in\Z^d} \vert a_{kn}\vert^p m_\lambda(k,n)^p\biggr)^{q/p}\biggr)^{1/q}.
\eeqsn
\end{Def}

\begin{Def}
\label{def1114G}
Let $F$ be a measurable function on $\R^{2d}$, and define
\beqsn
a_{kn}=\esssup_{(x,\xi)\in[0,1]^{2d}} \vert F(k+x,n+\xi)\vert.
\eeqsn
We say that $F\in W(L^{p,q}_{m_\lambda})$ if the sequence $a=(a_{kn})_{k,n\in\Z^d}$ belongs to $\ell^{p,q}_{m_\lambda}(\Z^{2d})$. The space $W(L^{p,q}_{m_\lambda})$ is called {\em amalgam space}, and has the norm defined by
\beqsn
\Vert F\Vert_{W(L^{p,q}_{m_\lambda})} = \Vert a \Vert_{\ell^{p,q}_{m_\lambda}}.
\eeqsn
\end{Def}

Let $\varphi\in\Sch_\omega(\R^d)$ and $\Lambda=\alpha_0\Z^d\times\beta_0\Z^d$ a
lattice with $\alpha_0,\beta_0>0$ sufficiently small so that
$\{\Pi(\sigma)\varphi\}_{\sigma\in\Lambda}$ is a Gabor frame for $L^2(\R^d)$. We indicate with $\widetilde{m}_\lambda$ the restriction of the weight \eqref{21bis} to the lattice $\Lambda$, in the sense that
\beqsn
\widetilde{m}_\lambda(k,n):=m_\lambda(\alpha_0 k,\beta_0 n).
\eeqsn
We recall the following result (see Proposition 11.1.4 of \cite{G}).
\begin{Prop}
\label{prop1114G}
Let $F\in W(L^{p,q}_{m_\lambda})$ be a continuous function, and $\alpha_0,\beta_0>0$. Then $F\vert_\Lambda\in \ell^{p,q}_{\widetilde{m}_\lambda}$, and there exists a constant $C=C(\alpha_0,\beta_0,\lambda)$ such that
\beqsn
\Vert F\vert_{\Lambda}\Vert_{\ell^{p,q}_{\widetilde{m}_\lambda}}\leq C\Vert F\Vert_{W(L^{p,q}_{m_\lambda})}.
\eeqsn
\end{Prop}

Now, we study the {\em Gabor frame operator} associated to the lattice $\Lambda$, given by
\beqs
\label{40}
S_{\varphi,\psi}f=\sum_{\sigma\in\Lambda}\langle f,\Pi(\sigma)\varphi\rangle
\Pi(\sigma)\psi,
\eeqs
for $\varphi,\psi,f\in L^2(\R^d)$.

We write as usual $S_{\varphi,\psi}=D_\psi C_\varphi$, where $C_\varphi$ is the `analysis' operator, acting on a function $f$ as
\beqs
\label{41}
C_\varphi f=\langle f,\Pi(\sigma)\varphi\rangle,\qquad \sigma\in\Lambda,
\eeqs
and $D_\psi$ is the `synthesis' operator, acting on a sequence $c=(c_{kn})_{k,n\in\Z^d}$ as
\beqs
\label{42}
D_\psi c = \sum_{k,n\in\Z^d} c_{kn}\Pi(\alpha_0 k,\beta_0 n)\psi.
\eeqs
We analyse the action of the previous operators on the modulation spaces $\boldsymbol{M}^{p,q}_{m_\lambda}$. The proofs of the next two results are very similar to \cite[Thms. 12.2.3, 12.2.4]{G}, so we omit them.
We just remark that, since $\varphi\in\Sch_\omega\subset\Sch$, we have that $V_\varphi\varphi\in\Sch$; then by Proposition 12.1.11 of \cite{G} we have $V_\varphi\varphi\in W(L^1_{v_\lambda})$, and so we can apply Theorem 11.1.5 of \cite{G}.

\begin{Th}
\label{th1223G}
Let $\varphi\in\Sch_\omega(\R^d)$ and $\Lambda$ a lattice as before. Then the operator
\beqsn
C_\varphi:\boldsymbol{M}^{p,q}_{m_\lambda}(\R^d)\longrightarrow \ell^{p,q}_{\widetilde{m}_\lambda}(\Z^{2d})
\eeqsn
is bounded for every $\lambda\in \R\setminus\{0\}$, $\alpha_0,\beta_0>0$, and $1\leq p,q\leq\infty$.
\end{Th}

\begin{Th}
\label{Th1224G}
Let $\psi\in\Sch_\omega(\R^d)$. Then we have:
\begin{itemize}
\item[(i)]
The operator
\beqsn
D_\psi:\ell^{p,q}_{\widetilde{m}_\lambda}(\Z^{2d})\longrightarrow \boldsymbol{M}^{p,q}_{m_\lambda}(\R^d)
\eeqsn
is bounded, for every $1\leq p,q\leq\infty$, $\alpha_0,\beta_0>0$, and $\lambda\in\R\setminus\{0\}$.
\item[(ii)]
For every $c\in\ell^{p^\prime,q^\prime}_{\widetilde{m}_{-\lambda}}$ and $f\in\boldsymbol{M}^{p,q}_{m_\lambda}$ we have that
\beqs
\label{52}
\langle D_\psi c,f\rangle = \langle c,C_\psi f\rangle,\qquad \text{for\ }1\leq p,q<\infty
\eeqs
and
\beqs
\label{53}
\langle C_\psi f,c\rangle = \langle f,D_\psi c\rangle,\qquad \text{for\ }1<p,q\leq\infty.
\eeqs
\item[(iii)]
For $p,q<\infty$, we have that $D_\psi c$ converges unconditionally in $\boldsymbol{M}^{p,q}_{m_\lambda}$; if $p=q=\infty$, then $D_\psi c$ converges unconditionally weak${}^*$ in $\boldsymbol{M}^{\infty}_{1/v_\lambda}$.
\end{itemize}
\end{Th}

Now,  we study the Gabor frame operator \eqref{40}. We recall (see \cite[Prop. 5.1.1 and 5.2.1]{G}) that
if we take a window $\varphi\in L^2(\R^d)$ and a lattice $\Lambda$ such that $\{\Pi(\sigma)\varphi\}_{\sigma\in\Lambda}$ is a Gabor frame for $L^2(\R^d)$, the operator \eqref{40} is invertible on $L^2(\R^d)$. Moreover, if we define the
{\em dual window} $\psi$ of $\varphi$ by $\psi:=S_{\varphi,\varphi}^{-1}\varphi$, we have that for every $f\in L^2(\R^d)$,
\beqsn
f=\sum_{\sigma\in\Lambda}\langle f,\Pi(\sigma)\varphi\rangle\Pi(\sigma)\psi
\eeqsn
with unconditional convergence in $L^2(\R^d)$. We observe also that
if $\varphi\in\Sch_\omega(\R^d)$ then
the dual window $\psi\in\Sch_\omega(\R^d)$ by \cite[Thm. 4.2]{GZ}.

\begin{Lemma}
\label{cor1226G}
Fix $\varphi\in\Sch_\omega(\R^d)\setminus\{0\}$, and let $\psi\in\Sch_\omega(\R^d)\setminus\{0\}$ be the dual window of $\varphi$. For $f\in \boldsymbol{M}^{p,q}_{m_\lambda}(\R^d)$, $\lambda\in\R\setminus\{0\}$, we have
\beqsn
f=D_\psi C_\varphi f = \sum_{\sigma\in\Lambda} \langle f,\Pi(\sigma)\varphi\rangle \Pi(\sigma)\psi
\eeqsn
and
\beqsn
f=D_\varphi C_\psi f = \sum_{\sigma\in\Lambda} \langle f,\Pi(\sigma)\psi\rangle \Pi(\sigma)\varphi,
\eeqsn
with convergence in $\boldsymbol{M}^{p,q}_{m_\lambda}$ for $p,q<\infty$, and weak${}^*$ convergence in $\boldsymbol{M}^{\infty}_{1/v_\lambda}$ in the case $p=q=\infty$.
\end{Lemma}

\begin{proof}
We first consider the case $p,q<\infty$. From Proposition \ref{Prop1134G} we have that there exists a sequence $f_n\in\Sch_\omega(\R^d)$ such that $f_n\to f$ in $\boldsymbol{M}^{p,q}_{m_\lambda}$ as $n\to \infty$. Since $\Sch_\omega(\R^d)\subset L^2(\R^d)$, we have that
\beqs
\label{57}
f_n=D_\psi C_\varphi f_n=D_\varphi C_\psi f_n.
\eeqs
From Theorems \ref{th1223G} and \ref{Th1224G} we obtain $D_\psi C_\varphi f_n \to D_\psi C_\varphi f$ and $D_\varphi C_\psi f_n\to D_\varphi C_\psi f$ in $\boldsymbol{M}^{p,q}_{m_\lambda}$, and so from \eqref{57} the result is proved. \\[0.2cm]
We now pass to the case $p=q=\infty$. Let $f\in\boldsymbol{M}^{\infty}_{1/v_\lambda}$ and $g\in \boldsymbol{M}^1_{v_\lambda}$. We have to prove that
\beqs
\label{58}
\langle f,g\rangle = \langle D_\psi C_\varphi f,g\rangle = \langle D_\varphi C_\psi f,g\rangle.
\eeqs
From \eqref{52} and \eqref{53} we have that
\beqsn
\langle D_\psi C_\varphi f,g\rangle = \langle f, D_\varphi C_\psi g\rangle;
\eeqsn
from the previous point we have that $D_\varphi C_\psi g=g$ in $\boldsymbol{M}^1_{v_\lambda}$, so the first equality in \eqref{58} is proved. The other is similar.
\end{proof}

\begin{Rem}
\begin{em}
\label{distributions}
Let $u\in\Sch_\omega^\prime(\R^d)$, and $\varphi,\psi\in\Sch_\omega(\R^d)$ as in Lemma \ref{cor1226G}. Then for every $\theta\in\Sch_\omega(\R^d)$ we have
\beqs
\label{24RW}
\langle u,\theta\rangle=\sum_{\sigma\in\Lambda}
\langle u,\Pi(\sigma)\varphi\rangle\langle\Pi(\sigma)\psi,\theta\rangle.
\eeqs
We have indeed that from Remark \ref{form25RW} there exists $\lambda<0$ such that $u\in \boldsymbol{M}^\infty_{m_\lambda} = \boldsymbol{M}^\infty_{1/v_\lambda}$. Then, from Lemma \ref{cor1226G}, for every $g\in \boldsymbol{M}^1_{v_\lambda}$,
\beqsn
\langle u,g\rangle = \sum_{\sigma\in\Lambda} \langle u,\Pi(\sigma) \varphi\rangle \langle\Pi(\sigma) \psi,g\rangle.
\eeqsn
From Proposition \ref{Prop1134G}, the previous formula then holds for $g=\theta\in\Sch_\omega(\R^d)$, so we have \eqref{24RW}.
\end{em}
\end{Rem}

We can now prove the main result of this section.
\begin{Th}
\label{th35RW}
If $u\in\Sch'_\omega(\R^d)$ then
\beqsn
\WF'_\omega(u)=\WF^G_\omega(u).
\eeqsn
\end{Th}

\begin{proof}
The inclusion $\WF^G_\omega(u)\subseteq \WF'_\omega(u)$ is trivial, so that we only have to prove that
\beqsn
\WF'_\omega(u)\subseteq\WF^G_\omega(u).
\eeqsn

Let $0\neq z_0\notin\WF^G_\omega(u)$. So, there exists an open conic set
$\Gamma\subset\R^{2d}\setminus\{0\}$ containing $z_0$ such that
\eqref{23} is satisfied.
By Remark \ref{distributions} we have that,
for $\varphi\in\Sch_\omega(\R^d)\setminus\{0\}$ and $\tilde\varphi=
S_{\varphi\varphi}^{-1}\varphi\in\Sch_\omega(\R^d)$ its dual window,
\beqsn
\langle u,\psi\rangle=\sum_{\sigma\in\Lambda}V_\varphi u(\sigma)
\langle\Pi(\sigma)\tilde\varphi,\psi\rangle\qquad\forall\psi\in\Sch_\omega(\R^d).
\eeqsn

We denote
\beqsn
u_1=&&\sum_{\sigma\in\Lambda\cap\Gamma}V_\varphi u(\sigma)\Pi(\sigma)\tilde\varphi,\\
u_2=&&\sum_{\sigma\in\Lambda\setminus\Gamma}V_\varphi u(\sigma)\Pi(\sigma)\tilde\varphi.
\eeqsn
Clearly $V_\varphi u(z)=V_\varphi u_1(z)+V_\varphi u_2(z)$.
Denoting $\sigma=(\sigma_1,\sigma_2)\in\R^d\times\R^d$, by \eqref{3}, \eqref{4},
the subadditivity of $\omega$ and \eqref{6}, we can estimate, for every $\alpha,\beta
\in\N_0^d$, $\lambda,\mu>0$:
\beqsn
\lefteqn{e^{-\lambda\varphi^*_\omega\left(\frac{|\alpha|}{\lambda}\right)}
e^{-\mu\varphi^*_\omega\left(\frac{|\beta|}{\mu}\right)}|x^\beta\partial^\alpha u_1(x)|}\\
&&\le\sum_{\sigma\in\Lambda\cap\Gamma}|V_\varphi u(\sigma)|\cdot
\big|x^\beta\partial^\alpha\big(e^{i\langle\sigma_2,x\rangle}
\tilde\varphi(x-\sigma_1)\big)\big|
e^{-\lambda\varphi^*_\omega\left(\frac{|\alpha|}{\lambda}\right)}
e^{-\mu\varphi^*_\omega\left(\frac{|\beta|}{\mu}\right)}\\
&&\le\sum_{\sigma\in\Lambda\cap\Gamma}|V_\varphi u(\sigma)|
\sum_{\gamma\leq\alpha}\binom\alpha\gamma2^{-|\alpha|}
|x|^{|\beta|}e^{-\mu\varphi^*_\omega\left(\frac{|\beta|}{\mu}\right)}
\langle\sigma_2\rangle^{|\alpha-\gamma|}|\partial^\gamma
\tilde\varphi(x-\sigma_1)|e^{-\lambda\varphi^*_\omega\left(\frac{|\alpha|}{\lambda}\right)}
2^{|\alpha|}\\
&&\le\sum_{\sigma\in\Lambda\cap\Gamma}|V_\varphi u(\sigma)|
\sum_{\gamma\leq\alpha}\binom\alpha\gamma2^{-|\alpha|}e^{\mu\omega(x)}
|\partial^\gamma\tilde\varphi(x-\sigma_1)|
\langle\sigma_2\rangle^{|\alpha-\gamma|}e^{3\lambda}e^{-3\lambda\varphi^*_\omega\left(\frac{|\alpha|}{3\lambda}\right)}\\
&&\le C_\lambda\sum_{\sigma\in\Lambda\cap\Gamma}|V_\varphi u(\sigma)|
\sum_{\gamma\leq\alpha}\binom\alpha\gamma2^{-|\alpha|}
e^{\mu\omega(\sigma_1)}e^{\mu\omega(x-\sigma_1)}|\partial^\gamma
\tilde\varphi(x-\sigma_1)|\langle\sigma_2\rangle^{|\alpha-\gamma|}
e^{-3\lambda\varphi^*_\omega\left(\frac{|\alpha|}{3\lambda}\right)}\\
&&\le C_{\lambda,\lambda',\mu}
\sum_{\sigma\in\Lambda\cap\Gamma}|V_\varphi u(\sigma)|
\sum_{\gamma\leq\alpha}\binom\alpha\gamma2^{-|\alpha|}e^{\mu\omega(\sigma_1)}
e^{\lambda' \varphi^*_\omega\left(\frac{|\gamma|}{\lambda'}\right)-3\lambda\varphi^*_\omega\left(\frac{|\alpha|}{3\lambda}\right)}\langle\sigma_2\rangle^{|\alpha-\gamma|}
\eeqsn
for some $C_\lambda,C_{\lambda,\lambda',\mu}>0$.

For $\lambda'=6\lambda$ we apply \cite[Prop. 2.1(g)]{BJ-Kotake}, then
\eqref{3} and \eqref{23}, and finally obtain, for some constants depending on $\lambda$ and
$\mu$:
\beqs
\nonumber
\lefteqn{e^{-\lambda\varphi^*_\omega\left(\frac{|\alpha|}{\lambda}\right)}
e^{-\mu\varphi^*_\omega\left(\frac{|\beta|}{\mu}\right)}
|x^\beta\partial^\alpha u_1(x)|}\\
\nonumber
&&\le C_{\lambda,6\lambda,\mu}
\sum_{\sigma\in\Lambda\cap\Gamma}|V_\varphi u(\sigma)|
\sum_{\gamma\leq\alpha}\binom\alpha\gamma2^{-|\alpha|}e^{\mu\omega(\sigma_1)}
e^{-6\lambda\varphi^*_\omega\left(\frac{|\alpha-\gamma|}{6\lambda}\right)}
\langle\sigma_2\rangle^{|\alpha-\gamma|}\\
\nonumber
&& \le C_{\lambda,6\lambda,\mu}
\sum_{\sigma\in\Lambda\cap\Gamma}|V_\varphi u(\sigma)|
\sum_{\gamma\leq\alpha}\binom\alpha\gamma2^{-|\alpha|}e^{\mu\omega(\sigma_1)}
e^{6\lambda\omega(\langle\sigma_2\rangle)}\\
\nonumber
&&\le C_{\lambda,\mu}\sum_{\sigma\in\Lambda\cap\Gamma}|V_\varphi u(\sigma)|
e^{(\mu+6\lambda)\omega(\langle\sigma\rangle)+\omega(\langle\sigma\rangle)}
e^{-\omega(\langle\sigma\rangle)}\\
\label{26}
&&\le C'_{\lambda,\mu}\sum_{\sigma\in\Lambda\cap\Gamma}
e^{-\omega(\langle\sigma\rangle)}
\leq C''_{\lambda,\mu},\qquad x\in\R^d.
\eeqs
This proves that $u_1\in\Sch_\omega(\R^d)$ (here, we consider the seminorms given in \eqref{seminorms}).
Therefore, from Theorem~\ref{th27GZ}, $V_\varphi u_1\in\Sch_\omega(\R^{2d})$ and
 for every $\lambda>0$ there is a constant $C_\lambda>0$
such that
\beqs
\label{32bis}
e^{\lambda\omega(z)}|V_\varphi u_1(z)|\leq C_\lambda\qquad\forall z\in\R^{2d}.
\eeqs

Let us now fix an open conic set $\Gamma'\subset\R^{2d}\setminus\{0\}$ containing
$z_0$ and such that $\overline{\Gamma'\cap S_{2d-1}}\subseteq \Gamma$.

Then
\beqs
\label{30}
\inf_{\afrac{0\neq\sigma\in\Lambda\setminus\Gamma}{z\in\Gamma'}}
\left|\frac{\sigma}{|\sigma|}-z\right|=\varepsilon>0
\eeqs
and $|\sigma-z|\geq\varepsilon|\sigma|$ for $0\neq\sigma\in\Lambda\setminus\Gamma$
and $z\in\Gamma'$.

From the subadditivity of $\omega$ we have
\beqs
\nonumber
e^{\lambda\omega(z)}|V_\varphi u_2(z)|\leq&&\sum_{\sigma\in\Lambda\setminus\Gamma}
e^{\lambda\omega(\sigma)+\lambda\omega(z-\sigma)}
|V_\varphi u_{2}(\sigma)|\cdot
|\langle\Pi(\sigma)\tilde\varphi,\Pi(z)\varphi\rangle|\\
\label{27}
\leq&&C\sum_{\sigma\in\Lambda\setminus\Gamma}
e^{(\lambda+\bar\lambda)\omega(\sigma)}
e^{\lambda\omega(z-\sigma)}|V_\varphi\tilde\varphi(z-\sigma)|,
\eeqs
for some $C,\bar\lambda>0$, because of Theorem~\ref{th24GZ} and since
(\cite[pg 41]{G})
\beqs
\label{28}
|\langle\Pi(\sigma)\tilde\varphi,\Pi(z)\varphi\rangle|
=|e^{-i\langle\sigma_1,z_2-\sigma_2\rangle}V_\varphi\tilde\varphi(z-\sigma)|
=|V_\varphi\tilde\varphi(z-\sigma)|.
\eeqs

Since $\tilde\varphi\in\Sch_\omega(\R^{d})$, from Theorem~\ref{th27GZ}
we have that for every $\mu>0$ there is a constant $C_\mu>0$ such that
\beqsn
|V_\varphi\tilde\varphi(z-\sigma)|\leq C_\mu e^{-\mu\omega(z-\sigma)}
\eeqsn
and hence, substituting in \eqref{27}:
\beqs
\label{29}
e^{\lambda\omega(z)}|V_\varphi u_2(z)|\leq CC_\mu
\sum_{\sigma\in\Lambda\setminus\Gamma}e^{(\lambda+\bar\lambda)\omega(\sigma)}
e^{(\lambda-\mu)\omega(z-\sigma)}.
\eeqs

However, for $z\in\Gamma'$ and $\sigma\in\Lambda\setminus\Gamma$ we have
$|\sigma-z|\geq\varepsilon|\sigma|$ and therefore, by the subadditivity of
$\omega$,
\beqsn
-\omega(z-\sigma)\leq-\omega(\varepsilon\sigma)\leq -M\omega(\sigma)
\eeqsn
for some $M>0$ depending on the constant $\varepsilon$ defined in \eqref{30}.
By formula \eqref{29} we  obtain
\beqs
\label{31}
e^{\lambda\omega(z)}|V_\varphi u_2(z)|\leq CC_\mu
\sum_{\sigma\in\Lambda\setminus\Gamma}e^{(\lambda+\bar\lambda+\lambda M-\mu M)
\omega(\sigma)}
\leq C_\lambda,\qquad z\in\Gamma',
\eeqs
for some $C_\lambda>0$, if $\mu$ is chosen large enough.

From \eqref{32bis} and \eqref{31} we finally deduce
\beqsn
\sup_{z\in\Gamma'}e^{\lambda\omega(z)}|V_\varphi u(z)|<
+\infty,\qquad\lambda>0,
\eeqsn
and hence $z_0\notin\WF'_\omega(u)$.
\end{proof}

From Theorem~\ref{th35RW}, in what follows we use $\WF'_\omega(u)$ for $\WF^G_\omega(u)$ and any $u\in\Sch'_{\omega}(\R^{d})$.
\begin{Prop}
\label{WFSomega}
For every $u\in\Sch^\prime_\omega(\R^d)$ we have $\WF'_\omega(u)=\emptyset$ if and only if $u\in\Sch_\omega(\R^d)$.
\end{Prop}
\begin{proof}
Suppose that $u\in\Sch_\omega(\R^d)$, and fix a window function $\varphi\in\Sch_\omega(\R^d)\setminus\{0\}$; then from Theorem \ref{th27GZ} we have that for every $\lambda>0$ there exists $C_\lambda>0$ such that
\beqsn
\vert V_\varphi u(z)\vert\leq C_\lambda e^{-\lambda\omega(z)},\qquad \forall z\in\R^{2d}.
\eeqsn
Then for every open conic set $\Gamma\subseteq\R^{2d}\setminus\{0\}$ condition \eqref{16} holds, so $\WF'_\omega(u)=\emptyset$. \hfill\break
Suppose now that $\WF'_\omega(u)=\emptyset$. From Definition \ref{def31RW} we have that for every $\overline{z}\in\R^{2d}\setminus\{0\}$ there exists an open conic set $\Gamma_{\overline{z}}\subseteq\R^{2d}\setminus\{0\}$ containing $\overline{z}$ such that for every $\lambda>0$ there exists $C_{\lambda,\overline{z}}>0$ satisfying
\beqsn
\vert V_\varphi u(z)\vert\leq C_{\lambda,\overline{z}} e^{-\lambda\omega(z)} \qquad\forall z\in\Gamma_{\overline{z}}.
\eeqsn
Let $\Upsilon_{\overline{z}} = \Gamma_{\overline{z}}\cap S_{2d-1}$. We have that $\{\Upsilon_{\overline{z}},\ \overline{z}\in\R^{2d}\setminus\{0\}\}$ is an open covering of $S_{2d-1}$; since $S_{2d-1}$ is compact and $\Gamma_{\overline{z}}$ is conic, there exist $z_1,\dots,z_k\in\R^{2d}\setminus\{0\}$ such that
\beqsn
\Gamma_{z_1}\cup\dots\cup\Gamma_{z_k}=\R^{2d}\setminus\{0\}.
\eeqsn
We then have that for every $\lambda>0$,
\beqsn
\vert V_\varphi u(z)\vert\leq C_\lambda e^{-\lambda\omega(z)} \qquad \forall z\in\R^{2d},
\eeqsn
where $C_\lambda=\max\{ C_{\lambda,z_1},\dots,C_{\lambda,z_k},\vert V_\varphi u(0)\vert e^{\lambda\omega(0)}\}$. From Theorem \ref{th27GZ} we finally have $u\in\Sch_\omega(\R^d)$.
\end{proof}

We now prove that the wave front set $\WF'_\omega$ is not affected by the phase-space shift operator.
\begin{Prop}
\label{cor52RW}
For every $w=(y,\eta)\in\R^{2d}$ and for every $u\in\Sch^\prime_\omega(\R^d)$ we have
\beqsn
\WF'_\omega(\Pi(w)u)=\WF'_\omega(u).
\eeqsn
\end{Prop}
\begin{proof}
Since $\Pi(w)=M_\eta T_y$, it is enough to prove that translation and modulation do not affect the wave front set. Concerning translation, we have that for $z=(x,\xi)\in\R^{2d}$,
\beqsn
V_\varphi (T_y u)(z)=\langle T_y u,\Pi(z)\varphi\rangle = \langle u,T_{-y}\Pi(z)\varphi\rangle = e^{-i\langle y,\xi\rangle} V_{T_{-y}\varphi}u;
\eeqsn
writing $\psi=T_{-y}\varphi\in\Sch_\omega(\R^d)$ we have that
\beqsn
\vert V_\varphi (T_y u)(z)\vert = \vert V_\psi u(z)\vert,
\eeqsn
and since the wave front set does not depend on the window (Proposition~\ref{cor33RW}) we have $\WF'_\omega(T_y u)=\WF'_\omega(u)$. Concerning modulation, we have
\beqsn
V_\varphi (M_\eta u)(z) = \langle M_\eta u,\Pi(z)\varphi\rangle = \langle u,M_{-\eta}\Pi(z)\varphi\rangle = e^{i\langle\eta,x\rangle} V_{M_{-\eta}\varphi}u(z);
\eeqsn
then, writing $\theta=M_{-\eta}\varphi\in\Sch_\omega(\R^d)$, we get
\beqsn
\vert V_\varphi (M_\eta u)(z)\vert = \vert V_\theta u(z)\vert,
\eeqsn
and as before we conclude that $\WF'_\omega(M_\eta u)=\WF'_\omega(u)$.
\end{proof}

The results obtained in Sections 2 and 3 are true with the weaker assumption (see Bj\"orck \cite{B}): ``there exist $a\in\R,b>0$ such that $\omega(t)\geq a+b\log(1+t)$ for $t\ge 0$'' instead of $(\gamma)$ of Definition~\ref{defweight}.
A detailed and modern treatment of these type of weights can be found \cite{BG}. Moreover, the results above are true in the quasi-analytic case also, i.e. when we consider that
$\omega(t)=o(t)$, as $t\to+\infty$, instead of condition $(\beta)$ of Definition~\ref{defweight}.

\section{Applications to (pseudo-)differential operators}
\label{sec4}

In this section we analyze the action of several  operators of pseudo-differential (or differential) type on the global wave front set $\WF'_\omega(u)$ of $u\in\Sch'_{\omega}(\R^{d})$. We will use the kernel theorem in $\Sch_{\omega}$. It is known that $\Sch_{\omega}$ is nuclear for many weight functions $\omega$. For example, whenever they satisfy the following condition:
\begin{equation}\label{BMM}
\exists\ H\ge 1\ \forall\ t\ge 0,\ 2\omega(t)\le \omega(Ht)+H.
\end{equation}
Morever, Bonet, Meise and Melikhov~\cite{BMM} proved that under such a condition the classes of ultradifferentiable functions defined by sequences in the sense of Komatsu satisfying the standard conditions $(M0)$, $(M1)$, $(M2)$ and $(M3)$, and the classes defined by weight functions in the sense of Braun, Meise and Taylor~\cite{BMT} coincide. Hence, under condition \eqref{BMM} our results are true also for spaces of the type we are treating defined by sequences (see, for instance, Langenbruch~\cite{L} for the definition of the spaces and several properties of them).

We start by defining the following symbol class.
\begin{Def}
\label{def23RW}
For $m\in\R$ we define
\beqsn
S^m_\omega:=\{&&\!\!\!a\in C^\infty(\R^{2d}):
\forall\lambda,\mu>0\ \exists C_{\lambda,\mu}>0
\ \mbox{such that}\\
&&|\partial_x^\alpha\partial_\xi^\beta a(x,\xi)|\leq C_{\lambda,\mu}
e^{\lambda\varphi^*_\omega\left(\frac{|\alpha|}{\lambda}\right)}
e^{\mu\varphi^*_\omega\left(\frac{|\beta|}{\mu}\right)}e^{m\omega(\xi)},\
\forall(x,\xi)\in\R^{2d},\alpha,\beta\in\N_0^d\}.
\eeqsn
\end{Def}

Then we consider the Kohn-Nirenberg quantization defined by
\beqs
\label{32}
a(x,D)f(x):=(2\pi)^{-d}\int_{\R^d}e^{i\langle x,\xi\rangle}a(x,\xi)\widehat{f}(\xi)d\xi,
\qquad a\in S^m_\omega, f\in\Sch_\omega(\R^d).
\eeqs
The above Kohn-Nirenberg quantization is well defined since
$\widehat f\in\Sch_\omega(\R^d)$ and hence for every $\lambda>0$ there exists
$C_\lambda>0$ such that
\beqsn
|a(x,\xi)|\cdot|\widehat f(\xi)|\leq e^{m\omega(\xi)}C_\lambda
e^{-\lambda\omega(\xi)}
\eeqsn
which is integrable in $\R^d$ if we choose $\lambda>0$ sufficiently large.
Moreover,
\beqsn
a(x,D):\ \Sch_\omega\longrightarrow\Sch'\subseteq\Sch'_\omega.
\eeqsn

If $\Sch_{\omega}$ is nuclear, we can use the kernel theorem and find a unique distribution
$K\in\Sch'_\omega(\R^{4d})$ of the linear operator
\beqsn
V_\varphi a(x,D)V_\varphi^*:\ \Sch_\omega(\R^{2d})\longrightarrow
\Sch'_\omega(\R^{2d})
\eeqsn
such that
\beqs
\label{Ad10}
V_\varphi a(x,D)V_\varphi^*F(y',\eta')
=(2\pi)^d\int_{\R^{2d}}K(y',\eta';y,\eta)F(y,\eta)dyd\eta
\qquad\forall F\in\Sch_\omega(\R^{2d}),
\eeqs
in the sense that
\beqs
\label{Ad11}
\langle V_\varphi a(x,D)V_\varphi^*F,G\rangle
=(2\pi)^d\langle K(y',\eta';y,\eta),G(y',\eta')F(y,\eta)\rangle
\qquad\forall G\in\Sch_\omega(\R^{2d}).
\eeqs

If $u\in\Sch_\omega(\R^d)$ and $F=V_\varphi u\in\Sch_\omega(\R^{2d})$ for
$\varphi\in\Sch_\omega(\R^d)$ with $\|\varphi\|_{L^2}=1$, then, from \eqref{Ad9},
\beqsn
V_\varphi a(x,D) u(y',\eta')=&&(2\pi)^{-d}V_\varphi a(x,D)V_\varphi^*
V_\varphi u(y',\eta')\\
=&&\int_{\R^{2d}}K(y',\eta';y,\eta)V_\varphi u(y,\eta)dyd\eta
\eeqsn
and we can compute the kernel directly:

\begin{Lemma}
\label{th1714NR}
For $a\in S^m_\omega$, $\varphi\in\Sch_\omega(\R^d)$ with $\|\varphi\|_{L^2}=1$
and $u\in\Sch_\omega(\R^d)$
we have that
\beqs
\label{311RW}
V_\varphi(a(x,D)u)(z')=\int_{\R^{2d}}K(z',z)V_\varphi u(z)dz,
\eeqs
where, for all $z=(y,\eta),z'=(y',\eta')\in\R^{2d}$,
\beqs
\label{312RW}
K(z',z)=(2\pi)^{-2d}e^{i\langle y,\eta\rangle}\int_{\R^{2d}}
e^{i(\langle x,\xi\rangle-\langle y,\xi\rangle-\langle x,\eta'\rangle)}
a(x,\xi)\widehat\varphi(\xi-\eta)\overline{\varphi(x-y')}dxd\xi.
\eeqs
\end{Lemma}

\begin{proof}
Let $F\in\Sch_\omega(\R^{2d})$ and consider the Kohn-Nirenberg quantization
\eqref{32} of $V_\varphi^*F\in\Sch_\omega(\R^{d})$:
\beqsn
a(x,D)V_\varphi^*F(x)=(2\pi)^{-d}\int_{\R^d}e^{i\langle x,\xi\rangle}a(x,\xi)
\widehat{V_\varphi^*F}(\xi)d\xi.
\eeqsn

Then, by \eqref{Ad1},
\beqsn
\lefteqn{V_\varphi a(x,D)V_\varphi^*F(y',\eta')=\int_{\R^d}(a(x,D)V_\varphi^*F)(x)
\overline{\varphi(x-y')}e^{-i\langle x,\eta'\rangle}dx}\\
&&=(2\pi)^{-d}\int_{\R^d}\int_{\R^d}e^{i\langle x,\xi\rangle}a(x,\xi)
\widehat{V_\varphi^*F}(\xi)\overline{\varphi(x-y')}e^{-i\langle x,\eta'\rangle}d\xi dx\\
&&=(2\pi)^{-d}\int_{\R^d}\int_{\R^d}\int_{\R^d}e^{i\langle x,\xi\rangle}a(x,\xi)
V_\varphi^*F(x')e^{-i\langle x',\xi\rangle}\overline{\varphi(x-y')}e^{-i\langle x,\eta'\rangle}
dx'd\xi dx\\
&&=(2\pi)^{-d}\int_{\R^d}\int_{\R^d}\int_{\R^d}\int_{\R^{2d}}
e^{i\langle x,\xi\rangle}a(x,\xi)F(y,\eta)e^{i\langle x',\eta\rangle}
\varphi(x'-y)\\
&&\hspace*{40mm}\cdot
e^{-i\langle x',\xi\rangle}\overline{\varphi(x-y')}e^{-i\langle x,\eta'\rangle}
dyd\eta dx'd\xi dx.
\eeqsn
Since $a\in S^m_\omega$, $F\in\Sch_\omega(\R^{2d})$ and
$\varphi\in\Sch_\omega(\R^d)$, we have that for every
$\lambda_1,\lambda_2,\lambda_3>0$ there exists a constant $C_{\lambda}>0$
such that
\beqsn
\lefteqn{|a(x,\xi)F(y,\eta)\varphi(x'-y)\overline{\varphi(x-y')}|}\\
&&\leq C_{\lambda}e^{m\omega(\xi)}
e^{-\lambda_1\omega(y,\eta)}e^{-\lambda_2\omega(x'-y)}
e^{-\lambda_3\omega(x-y')}\\
&&\leq C_{\lambda}e^{m\omega(\xi)}
e^{-\frac{\lambda_1}{2}\omega(y)}e^{-\frac{\lambda_1}{2}\omega(\eta)}
e^{-\lambda_2\omega(x')+\lambda_2\omega(y)}
e^{-\lambda_3\omega(x)+\lambda_3\omega(y')}.
\eeqsn

Choosing $\lambda_1,\lambda_2>0$ sufficiently large we can apply Fubini's theorem
with respect to the variables $y,\eta$ and $x'$, obtaining:
\beqsn
V_\varphi a(x,D)V_\varphi^*F(y',\eta')=\\
=(2\pi)^{-d}\int_{\R^d}\int_{\R^d}\int_{\R^{2d}}&&
e^{i\langle x,\xi\rangle}a(x,\xi)F(y,\eta)\\
&&\cdot\left(\int_{\R^d}e^{i\langle x',\eta\rangle}\varphi(x'-y)
e^{-i\langle x',\xi\rangle}dx'\right)\overline{\varphi(x-y')}e^{-i\langle x,\eta'\rangle}
dyd\eta d\xi dx\\
=(2\pi)^{-d}\int_{\R^d}\int_{\R^d}\int_{\R^{2d}}&&
e^{i\langle x,\xi\rangle}a(x,\xi)F(y,\eta)\\
&&\cdot\left(\int_{\R^d}e^{i\langle y+s,\eta\rangle}e^{-i\langle y+s,\xi\rangle}\varphi(s)
ds\right)\overline{\varphi(x-y')}e^{-i\langle x,\eta'\rangle}
dyd\eta d\xi dx\\
=(2\pi)^{-d}\int_{\R^d}\int_{\R^d}\int_{\R^{2d}}&&
e^{i\langle x,\xi\rangle}a(x,\xi)F(y,\eta)e^{i\langle y,\eta\rangle}
e^{-i\langle y,\xi\rangle}\\
&&\cdot\left(\int_{\R^d}e^{-i\langle s,\xi-\eta\rangle}\varphi(s)
ds\right)\overline{\varphi(x-y')}e^{-i\langle x,\eta'\rangle}
dyd\eta d\xi dx\\
=(2\pi)^{-d}\int_{\R^d}\int_{\R^d}\int_{\R^{2d}}&&
e^{i\langle x,\xi\rangle}a(x,\xi)F(y,\eta)e^{i\langle y,\eta\rangle}
e^{-i\langle y,\xi\rangle}\\
&&\cdot\,\widehat{\varphi}(\xi-\eta)\overline{\varphi(x-y')}e^{-i\langle x,\eta'\rangle}
dyd\eta d\xi dx.
\eeqsn

Since $a\in S^m_\omega$, $F\in\Sch_\omega(\R^{2d})$ and
$\varphi\in\Sch_\omega(\R^d)$, for every
$\mu_1,\mu_2,\mu_3>0$ there exists a constant $C_{\mu}>0$
such that
\beqsn
\lefteqn{|a(x,\xi)F(y,\eta)\widehat{\varphi}(\xi-\eta)\overline{\varphi(x-y')}|}\\
&&\leq C_{\mu}e^{m\omega(\xi)}
e^{-\mu_1\omega(y)}e^{-\mu_1\omega(\eta)}e^{-\mu_2\omega(\xi)+\mu_2\omega(\eta)}
e^{-\mu_3\omega(x)+\mu_3\omega(y')},
\eeqsn
so that, for $\mu_3,\mu_1>\mu_2$ sufficiently large, we can apply Fubini's
theorem and obtain
\beqsn
V_\varphi a(x,D)V_\varphi^*F(y',\eta')=\\
=(2\pi)^{-d}\int_{\R^{2d}}&&F(y,\eta)e^{i\langle y,\eta\rangle}\\
&&\cdot\left(\int_{\R^{2d}}
e^{i(\langle x,\xi\rangle-\langle y,\xi\rangle-\langle x,\eta'\rangle)}
a(x,\xi)\widehat{\varphi}(\xi-\eta)\overline{\varphi(x-y')}dxd\xi
\right)dyd\eta.
\eeqsn

Applying the above result to $F=V_\varphi u$ for some $u\in\Sch_\omega(\R^d)$,
since $\|\varphi\|_{L^2}=1$ and hence
$V_\varphi^*F=V_\varphi^*V_\varphi u=(2\pi)^du$ by \eqref{Ad9},
we have
\beqsn
V_\varphi(a(x,D)u)(y',\eta')=\int_{\R^{2d}}K(y',\eta';y,\eta)V_\varphi u(y,\eta)dyd\eta,
\eeqsn
for
\beqsn
K(y',\eta';y,\eta)=(2\pi)^{-2d}e^{i\langle y,\eta\rangle}\int_{\R^{2d}}
e^{i(\langle x,\xi\rangle-\langle y,\xi\rangle-\langle x,\eta'\rangle)}
a(x,\xi)\widehat\varphi(\xi-\eta)\overline{\varphi(x-y')}dxd\xi,
\eeqsn
which concludes the proof of the lemma.
\end{proof}
In the next result the following property on the weight function $\omega$ will be useful: from \cite[Lemma 4.7(ii)]{BJO} (for instance), for every $\mu>0$ and $t\geq1$,
   \beqs
   \label{47ii}
   \inf_{\beta\in\N_o^d}t^{-|\beta|}e^{\mu\varphi^*_\omega\left(\frac{|\beta|}{\mu}\right)}
   \leq e^{-(\mu-\frac1b)\omega(t)-\frac ab},
   \eeqs
   where $a\in\R$ and $b>0$ are constants that depend on $\omega$.
\begin{Prop}
\label{prop37RW}
If $a\in S^m_\omega$, $m\in\R$  and $K\in C^\infty(\R^{4d})$ is defined by \eqref{312RW},
then for every $\lambda>0$ there exists a constant $C_\lambda>0$ such that
\beqs
\label{313RW}
|K(z',z)|\leq C_{\lambda} e^{-\lambda\omega(y-y')}
e^{(m-\lambda)\omega(\eta-\eta')}e^{m\omega(\eta')},
\qquad z=(y,\eta),z'=(y',\eta')\in\R^{2d}.
\eeqs
Moreover, if 
$a(z)=0$ for
$z\in\Gamma\setminus\overline{B(0,R)}$ for an open conic set
$\Gamma\subseteq\R^{2d}\setminus\{0\}$ and
for some $R>0$ (here $B(0,R)$ is the ball
of center $0$ and radius $R$ in $\R^{2d}$), then for every open conic set
$\Gamma'\subseteq\R^{2d}\setminus\{0\}$ such that
$\overline{\Gamma'\cap S_{2d-1}}\subseteq\Gamma$ we have that for every
$\lambda>0$ there exists a constant $C_\lambda>0$ such that
for all $z'=(y',\eta')\in\Gamma'$ and $z=(y,\eta)\in\R^{2d}$,
\beqs
\label{314RW}
|K(z',z)|\leq C_\lambda e^{-\lambda\omega(y-y')}e^{-\lambda\omega(\eta-\eta')}
e^{-2\lambda\omega(y')}e^{-2\lambda\omega(\eta')}.
\eeqs
\end{Prop}

\begin{proof}
By the linear change of variables $\xi'=\xi-\eta$ and $x'=x-y'$ in
\eqref{312RW} we have
\beqsn
K(z',z)=&&(2\pi)^{-2d}e^{i\langle y,\eta\rangle}\int_{\R^{2d}}
e^{i(\langle x'+y',\xi'+\eta\rangle-\langle y,\xi'+\eta\rangle
-\langle x'+y',\eta'\rangle)}a(x'+y',\xi'+\eta)
\widehat\varphi(\xi')\overline{\varphi(x')}dx'd\xi'\\
=&&(2\pi)^{-2d}e^{i(\langle y',\eta\rangle-\langle y',\eta'\rangle)}\\
&&\cdot
\int_{\R^{2d}}e^{i(\langle x',\xi'\rangle+\langle x',\eta\rangle+
\langle y',\xi'\rangle
-\langle y,\xi'\rangle-\langle x',\eta'\rangle)}
a(x'+y',\xi'+\eta)\widehat{\varphi}(\xi')
\overline{\varphi(x')}dx'd\xi',
\eeqsn
and hence, setting $x=x'$ and $\xi=\xi'$:
\beqs
\label{33}
|K(z',z)|=(2\pi)^{-2d}\left|\int_{\R^{2d}}
e^{i(\langle x,\eta-\eta'+\xi\rangle+\langle\xi,y'-y\rangle)}
a(x+y',\xi+\eta)\widehat\varphi(\xi)\overline{\varphi(x)}dxd\xi\right|.
\eeqs

Writing, for $M,N\in\N_0$,
\beqsn
e^{i(\langle x,\eta-\eta'+\xi\rangle+\langle\xi,y'-y\rangle)}
=&&\langle\eta-\eta'+\xi\rangle^{-2M}(1-\Delta_x)^M
e^{i(\langle x,\eta-\eta'+\xi\rangle+\langle\xi,y'-y\rangle)}\\
=&&\langle y-y'\rangle^{-2N}\langle\eta-\eta'+\xi\rangle^{-2M}(1-\Delta_x)^M
e^{i\langle x,\eta-\eta'+\xi\rangle}(1-\Delta_\xi)^Ne^{i\langle\xi,y'-y\rangle}
\eeqsn
 and integrating  by parts in \eqref{33}, we have
\beqs
\label{315RW}
|K(z',z)|=(2\pi)^{-2d}\langle y-y'\rangle^{-2N}
\left|\int_{\R^{2d}}e^{i(\langle x,\eta-\eta'\rangle+\langle\xi,y'-y\rangle)}
\lambda_{N,M}(y',\eta',\eta,x,\xi)dxd\xi\right|,
\eeqs
where
\beqsn
\lefteqn{\lambda_{N,M}(y',\eta',\eta,x,\xi)}\\
&&=(1-\Delta_\xi)^N\left[e^{i\langle x,\xi\rangle}
\langle\eta-\eta'+\xi\rangle^{-2M}
(1-\Delta_x)^M\left(a(x+y',\xi+\eta)\widehat\varphi(\xi)
\overline{\varphi(x)}\right)\right].
\eeqsn

For $a\in S^m_\omega$, since $\varphi,\widehat\varphi\in\Sch_\omega(\R^d)$,
by \cite[Thm. 4.8(5)]{BJO} we have that for every
$\lambda$, $\mu$, $\lambda'$, $\mu'$, $\lambda''$, $\mu''>0$
there are positive constants
$C_{\lambda,\mu},C_{\lambda',\lambda''},C_{\mu',\mu''}$, depending only on the indexed
constants, such that for every $M,N,k,\ell\in\N_0$:
\beqs
\nonumber
|\lambda_{N,M}(y',\eta',\eta,x,\xi)|\leq&&
\sum_{\gamma_1+\gamma_2+\gamma_3+\gamma_4=2N}
\frac{(2N)!}{\gamma_1!\gamma_2!\gamma_3!\gamma_4!}
\sum_{\sigma_1+\sigma_2=2M}\frac{(2M)!}{\sigma_1!\sigma_2!}
\langle x\rangle^{|\gamma_1|}\langle \eta-\eta'+\xi\rangle^{-2M-|\gamma_2|}\\
\nonumber
&&\cdot C_{\lambda,\mu}e^{\lambda\varphi^*_\omega\left(\frac{|\gamma_3|}{\lambda}\right)}
e^{\mu\varphi^*_\omega\left(\frac{|\sigma_1|}{\mu}\right)}
e^{m\omega(\xi+\eta)}
C_{\lambda',\lambda''}\langle\xi\rangle^{-k}
e^{\lambda'\varphi^*_\omega\left(\frac{k}{\lambda'}\right)}
e^{\lambda''\varphi^*_\omega\left(\frac{|\gamma_4|}{\lambda''}\right)}\\
\label{36}
&&\cdot
C_{\mu,\mu'}\langle x\rangle^{-\ell}
e^{\mu'\varphi^*_\omega\left(\frac{\ell}{\mu'}\right)}
e^{\mu''\varphi^*_\omega\left(\frac{|\sigma_2|}{\mu''}\right)}.
\eeqs

Note that
\beqs
\label{35}
\langle\eta-\eta'+\xi\rangle^{-1}\leq\sqrt{2}\langle\eta-\eta'\rangle^{-1}
\langle\xi\rangle.
\eeqs
Moreover, we can choose $\lambda''=\lambda$, $\mu''=\mu$ and apply
Proposition~2.1(g) of \cite{BJ-Kotake}. Taking also into account the
subadditivity of $\omega$, we have that for every $\lambda$, $\mu$, $\lambda'$,
$\mu'>0$ there is a constant $C_{\lambda,\mu,\lambda',\mu'}>0$
such that for all $M,N,k,\ell\in\N_0$:
\beqsn
|\lambda_{N,M}(y',\eta',\eta,x,\xi)|\leq&&C_{\lambda,\mu,\lambda',\mu'}
\sum_{\gamma_1+\gamma_2+\gamma_3+\gamma_4=2N}
\frac{(2N)!}{\gamma_1!\gamma_2!\gamma_3!\gamma_4!}2^{-2N}
\sum_{\sigma_1+\sigma_2=2M}\frac{(2M)!}{\sigma_1!\sigma_2!}2^{-2M}\\
&&\cdot
\langle x\rangle^{|\gamma_1|-\ell}\langle \eta-\eta'\rangle^{-2M-|\gamma_2|}
\langle\xi\rangle^{2M+|\gamma_2|}e^{m\omega(\xi)}
e^{m\omega(\eta-\eta')}e^{m\omega(\eta')}\\
&&\cdot
e^{\frac\lambda2\varphi^*_\omega\left(\frac{|\gamma_3+\gamma_4|}{\lambda/2}\right)}
2^{2N}
e^{\frac\mu2\varphi^*_\omega\left(\frac{|\sigma_1+\sigma_2|}{\mu/2}\right)}
2^{2M}
\langle\xi\rangle^{-k}
e^{\lambda'\varphi^*_\omega\left(\frac{k}{\lambda'}\right)}
e^{\mu'\varphi^*_\omega\left(\frac{\ell}{\mu'}\right)}.
\eeqsn

Taking the infimum on $k\in\N_0$ and applying \eqref{47ii} and
\eqref{4}, we get:
\beqsn
|\lambda_{N,M}(y',\eta',\eta,x,\xi)|\leq&&C_{\lambda,\mu,\lambda',\mu'}
\sum_{\gamma_1+\gamma_2+\gamma_3+\gamma_4=2N}
\frac{(2N)!}{\gamma_1!\gamma_2!\gamma_3!\gamma_4!}2^{-2N}
\sum_{\sigma_1+\sigma_2=2M}\frac{(2M)!}{\sigma_1!\sigma_2!}2^{-2M}\\
&&\cdot
e^{\mu'\varphi^*_\omega\left(\frac{\ell}{\mu'}\right)}\langle x\rangle^{2N-\ell}
\langle\eta-\eta'\rangle^{-2M}
e^{\frac\mu6\varphi^*_\omega\left(\frac{2M}{\mu/6}\right)}e^{\mu/2}\\
&&\cdot
\langle\xi\rangle^{2M+2N}e^{m\omega(\xi)}e^{-\left(\lambda'-\frac1b\right)
\omega(\xi)-\frac ab}
e^{m\omega(\eta-\eta')}e^{m\omega(\eta')}
e^{\frac\lambda6\varphi^*_\omega\left(\frac{2N}{\lambda/6}\right)}e^{\lambda/2}.
\eeqsn

Substituting in \eqref{315RW} we have that for all $\lambda,\mu,\lambda',\mu'>0$ there
is a constant $C'_{\lambda,\mu,\lambda',\mu'}>0$ such that for every $M,
N,\ell\in\N_0$:
\beqs
\nonumber
|K(z',z)|\leq&& C'_{\lambda,\mu,\lambda',\mu'}
\langle y-y'\rangle^{-2N}e^{\frac\lambda6\varphi^*_\omega\left(\frac{2N}{\lambda/6}\right)}
\langle\eta-\eta'\rangle^{-2M}e^{\frac\mu6\varphi^*_\omega\left(\frac{2M}{\mu/6}\right)}
e^{m\omega(\eta-\eta')}e^{m\omega(\eta')}\\
\label{34}
&&\cdot e^{\mu'\varphi^*_\omega\left(\frac{\ell}{\mu'}\right)}
\int_{\R^d}\langle x\rangle^{2N-\ell}dx
\int_{\R^d}\langle\xi\rangle^{2M+2N}e^{\left(m-\lambda'+\frac1b\right)\omega(\xi)}d\xi.
\eeqs

Let us now fix $\mu'>0$, choose $\ell\in\N$ and $\lambda'>0$ sufficiently large
so that the above integrals converge, take the infimum on $M$ and $N$, apply
\eqref{47ii} and obtain:
\beqs
\label{34bis}
|K(z',z)|\leq C_{\lambda,\mu}e^{-\left(\frac\lambda6-\frac1b\right)\omega(y-y')}
e^{-\left(\frac\mu6-\frac1b\right)\omega(\eta-\eta')}
e^{m\omega(\eta-\eta')}e^{m\omega(\eta')}.
\eeqs

In particular, by the arbitrariness of $\lambda$ and $\mu$ in \eqref{34bis}, we
have that
for every $\lambda,\mu>0$ there is a constant $C_{\lambda,\mu}>0$ such that
\beqs
\label{317RW}
|K(z',z)|\leq C_{\lambda,\mu} e^{-\lambda\omega(y-y')}
e^{(m-\mu)\omega(\eta-\eta')}e^{m\omega(\eta')}
\qquad\forall z=(y,\eta),z'=(y',\eta')\in\R^{2d},
\eeqs
which proves \eqref{313RW} for $\mu=\lambda$.

Applying \eqref{35} only to $\langle\eta-\eta'+\xi\rangle^{-M}$ in \eqref{36},
by the same computations as to get \eqref{34} we obtain that if $a(z)=0$ for
$z\in\Gamma\setminus\overline{B(0,R)}$, then
\beqs
\nonumber
|K(z',z)|\leq&&C'_{\lambda,\mu,\lambda',\mu'}\langle y-y'\rangle^{-2N}
e^{\frac\lambda6\varphi^*_\omega\left(\frac{2N}{\lambda/6}\right)}
\langle\eta-\eta'\rangle^{-M}e^{\frac\mu6\varphi^*_\omega\left(\frac{2M}{\mu/6}\right)}
e^{m\omega(\eta-\eta')}e^{m\omega(\eta')}\\
\label{37}
&&\cdot\int_{D_{y',\eta}}\langle\eta'-(\xi+\eta)\rangle^{-M}
e^{\mu'\varphi^*_\omega\left(\frac{\ell}{\mu'}\right)}
\langle x\rangle^{2N-\ell}\langle\xi\rangle^{M+2N}
e^{\left(m-\lambda'+\frac1b\right)\omega(\xi)}dxd\xi,
\eeqs
where
\beqsn
D_{y',\eta}:=\{(x,\xi)\in\R^{2d}:\ (x+y',\xi+\eta)\in
(\R^{2d}\setminus\Gamma)\cup\overline{B(0,R)}\}.
\eeqsn

We now want to estimate \eqref{37} for $z'=(y',\eta')\in\Gamma'$
and $z=(y,\eta)\in\R^{2d}$.
It has been proved in \cite[pg 643]{RW} that
\beqs
\label{38}
\langle y'\rangle\langle\eta'\rangle
\leq C\langle x\rangle^2\langle\eta'-(\xi+\eta)\rangle^2
\qquad\forall z'\in\Gamma'\setminus B(0,2R), z\in\R^{2d}, (x,\xi)\in D_{y',\eta}
\eeqs
for some constant $C>0$.

Substituting \eqref{38} into \eqref{37} and applying \cite[Prop. 2.1(g)]{BJ-Kotake} we have,
for $z'\in\Gamma'\setminus B(0,2R)$ and  $z\in\R^{2d}$:
\beqsn
|K(z',z)|\leq&&C^{M/2}C'_{\lambda,\mu,\lambda',\mu'}
\langle y-y'\rangle^{-2N}e^{\frac\lambda6\varphi^*_\omega\left(\frac{2N}{\lambda/6}\right)}\\
&&\cdot \langle\eta-\eta'\rangle^{-M}
e^{\frac{\mu}{12}\varphi^*_\omega\left(\frac{M}{\mu/12}\right)}
e^{\frac{\mu}{24}\varphi^*_\omega\left(\frac{M/2}{\mu/24}\right)}
\langle y'\rangle^{-M/2}
e^{\frac{\mu}{24}\varphi^*_\omega\left(\frac{M/2}{\mu/24}\right)}
\langle\eta'\rangle^{-M/2}\\
&&\cdot e^{m\omega(\eta-\eta')}e^{m\omega(\eta')}
\int_{D_{y',\eta}}\langle x\rangle^{M+2N-\ell}e^{\mu'\varphi^*_\omega\left(\frac{\ell}{\mu'}\right)}\langle\xi\rangle^{M+2N}
e^{\left(m-\lambda'+\frac1b\right)\omega(\xi)}dxd\xi.
\eeqsn

We now fix $\mu'>0$ and choose $\ell\in\N$ and $\lambda'>0$ sufficiently large
so that the above integral is convergent; then take the infimum on $M$ and $N$
and apply \eqref{47ii}. We obtain that for every $\lambda,\mu>0$ there is a constant
$C_{\lambda,\mu}>0$ such that
\beqsn
|K(z',z)|\leq&&C_{\lambda,\mu}e^{-\left(\frac\lambda6-\frac1b\right)\omega(y-y')}
e^{-\left(\frac{\mu}{12}-\frac1b\right)\omega(\eta-\eta')}\\
&&\cdot e^{-\left(\frac{\mu}{24}-\frac1b\right)\omega(y')}
e^{-\left(\frac{\mu}{24}-\frac1b\right)\omega(\eta')} e^{m\omega(\eta-\eta')}
e^{m\omega(\eta')}
\qquad\forall z'\in\Gamma'\setminus B(0,2R), z\in\R^{2d}.
\eeqsn

In particular, for $\bar\lambda=\frac\lambda6-\frac1b$ and $\bar\mu=\frac{\mu}{12}-\frac2b$
we have that there is $C_{\bar\lambda,\bar\mu}>0$ such that, for
$z'\in\Gamma'\setminus B(0,2R)$ and $z\in\R^{2d}$:
\beqsn
|K(z',z)|\leq C_{\bar\lambda,\bar\mu}e^{-\bar\lambda\omega(y-y')}
e^{(m-\bar\mu)\omega(\eta-\eta')}e^{\left(m-\frac{\bar\mu}{2}\right)\omega(\eta')}
e^{-\frac{\bar\mu}{2}\omega(y')}.
\eeqsn

For $\bar\mu\geq4\bar\lambda+2m$ we have that $m-\bar\mu\leq -\bar\lambda$ and
$-\frac{\bar\mu}{2}\leq m-\frac{\bar\mu}{2}\leq-2\bar\lambda$ which proves \eqref{314RW}
for $z'\in\Gamma'\setminus B(0,2R), z\in\R^{2d}$.

Since the estimate \eqref{314RW} for $|z'|\leq 2R$ follows from \eqref{317RW},
the proof is complete.
\end{proof}

\begin{Rem}
\begin{em}
For $a\in S^m_\omega$, $m\in\R,$ and $K\in C^\infty(\R^{4d})$ defined by \eqref{312RW}
the integral in \eqref{311RW} is well defined also for $u\in\Sch'_\omega(\R^d)$. In fact,  \eqref{Ad7} and \eqref{313RW} imply that there exist
$\tilde{C},\tilde{\lambda}>0$ and that for every $\lambda>0$ there exists
$C_\lambda>0$ such that
\beqs
\nonumber
|K(z',z)V_\varphi u(z)|\leq&&\tilde{C}C_\lambda
e^{-\lambda\omega(y-y')+(m-\lambda)\omega(\eta-\eta')}
e^{m\omega(\eta')}
e^{\tilde{\lambda}\omega(y)+\tilde{\lambda}\omega(\eta)}\\
\label{Ad17}
\leq&&\tilde{C}C_\lambda
e^{\lambda\omega(y')+(m+\lambda)\omega(\eta')}
e^{(\tilde{\lambda}-\lambda)\omega(y)+(m+\tilde{\lambda}-\lambda)\omega(\eta)}
\in L^1(\R^{2d}_{z=(y,\eta)})
\eeqs
if $\lambda>\max\{\tilde{\lambda},m+\tilde{\lambda}\}$.
\end{em}
\end{Rem}

We now want to extend Lemma~\ref{th1714NR} for $u\in\Sch'_\omega(\R^d)$. To this aim we first need
the next two results.

\begin{Prop}
\label{propdensity}
The space $\Sch_\omega(\R^d)$ is dense in $\Sch'_\omega(\R^d)$.
\end{Prop}

\begin{proof}
Let us consider the inclusion
\beqsn
i:\  \Sch_\omega(\R^d)&&\hookrightarrow \Sch'_\omega(\R^d)\\
f&&\mapsto\langle i(f),\varphi\rangle:=\int_{\R^d}f(x)\overline{\varphi(x)}dx
\quad\forall\varphi\in\Sch_\omega(\R^d).
\eeqsn

To show that the image is dense we take $T\in\left(\Sch'_\omega(\R^d)\right)'$ such that
$\left.T\right|_{\Sch_\omega}=0$ and prove that $T\equiv0$.

Since $\Sch_\omega(\R^d)$ is reflexive, there exists a unique
$f\in\Sch_\omega(\R^d)$ such that
\beqsn
T(\varphi)=\int_{\R^d}f(x)\overline{\varphi(x)}dx=0,
\qquad\forall\varphi\in\Sch_\omega(\R^d),
\eeqsn
because of $\left.T\right|_{\Sch_\omega}=0$.
Therefore $f=0$, i.e. $T\equiv0$.
\end{proof}

\begin{Prop}
\label{propcontS'}
Let $\varphi\in\Sch_\omega(\R^d)\setminus\{0\}$. Then
\beqsn
V_\varphi:\ \Sch'_\omega(\R^d)\longrightarrow\Sch'_\omega(\R^{2d})
\eeqsn
is continuous.
\end{Prop}

\begin{proof}
We already know that
\beqsn
V_\varphi^*:\ \Sch_\omega(\R^{2d})\longrightarrow\Sch_\omega(\R^{d})
\eeqsn
is continuous by \eqref{Eq1}. It follows that
\beqsn
(V_\varphi^*)^*:\ \Sch'_\omega(\R^d)\longrightarrow\Sch'_\omega(\R^{2d})
\eeqsn
is continuous and moreover
$\left.(V_\varphi^*)^*\right|_{\Sch_\omega(\R^d)}=V_\varphi$ because, for
$f,g\in\Sch_\omega(\R^d)$,
\beqsn
\langle(V_\varphi^*)^* f,g\rangle=\langle f,V_\varphi^* g\rangle
=\langle V_\varphi f,g\rangle.
\eeqsn

Since $\Sch_\omega(\R^d)$ is dense in $\Sch'_\omega(\R^d)$ by
Proposition~\ref{propdensity}, we have that $(V_\varphi^*)^*$ is the
continuous extension of $V_\varphi$ to $\Sch'_\omega(\R^d)$ and, hence, $V_\varphi$ is continuous on $\Sch'_\omega(\R^d)$ also.
\end{proof}

Now, we need amplitudes $a(x,y,\xi)$, instead of symbols $a(x,\xi)$.

\begin{Def}
\label{amplitude}
Given $m\in\R$, we say that $a(x,y,\xi)\in C^{\infty}(\R^{3d})$ is an amplitude  in the space $S^{m}_{\omega}$ if for every $\lambda,\mu>0$ there is $C_{\lambda,\mu}>0$ such that
$$
|\partial_{x}^{\alpha}\partial_{y}^{\gamma}\partial_{\xi}^{\beta} a(x,y,\xi)|\le C_{\lambda,\mu} e^{\lambda \varphi^{*}(\frac{|\alpha+\gamma|}{\lambda})+\mu\varphi^{*}(\frac{|\beta|}{\mu})}e^{m\omega(\xi)},
$$
for all $(x,y,\xi)\in \R^{3d}$ and $\alpha,\beta,\gamma\in\N_{0}^{d}$.
\end{Def}

Now, proceeding in a similar way to that of Proposition 1.9 and Theorem 2.2 of \cite{FGJ}, one can prove that if $a(x,y,\xi)\in S^{m}_{\omega}$ is an amplitude as in Definition~\ref{amplitude}, the operator acting on $\Sch_{\omega}$, given by the iterated integral
$$
A(f)(x):=\int_{\R^{d}}\left(\int_{\R^{d}}e^{i\langle x-y,\xi\rangle}a(x,y,\xi)f(y)dy\right)d\xi, \quad f\in\Sch_{\omega},
$$
is well defined and continuous from $\Sch_{\omega}$ into itself. The operator $A$ is called \emph{pseudo-differential operator} of type $\omega$ with amplitude $a(x,y,\xi)$. Moreover, $A$ can be extended continuously to the dual space $\tilde{A}:\Sch'_{\omega}\to\Sch'_{\omega}$ in a standard way (see \cite[Theorem 2.5]{FGJ}). In particular, the Kohn-Nirenberg quantization defined in \eqref{32} is a pseudo-differential operator with amplitude $$a(x,y,\xi):=(2\pi)^{-d}p(x,\xi),$$
where $p(x,\xi)$ is a symbol as in Definition~\ref{def23RW}.

As a consequence of the above considerations and of the estimates of the kernel in
Proposition~\ref{prop37RW}, we obtain the following result:

\begin{Cor}
\label{corkernel}
Let $a(x,\xi)\in S^m_\omega$ a symbol as in Definition~\ref{def23RW}, $\varphi\in\Sch_\omega(\R^d)$ with $\|\varphi\|_{L^2}=1$
and $u\in\Sch'_\omega(\R^d)$.
Then, for $K(z',z)$ as in \eqref{312RW}, we have
\beqs
\label{Ad18}
V_\varphi a(x,D)u(z')=\int_{\R^{2d}}K(z',z)V_\varphi u(z)dz,
\eeqs
for all $z'\in\R^{2d}$.
\end{Cor}

\begin{proof}
Since
$V_\varphi$ operates on $\Sch'_\omega$, from the previous comments it is clear that $V_\varphi a(x,D)$ can be extended to $\Sch'_\omega(\R^d)$. We take $u\in\Sch'_\omega(\R^d)$. By Proposition~\ref{propdensity}, there exists a sequence
$\{u_n\}_{n\in\N}\subset\Sch_\omega(\R^d)$ which converges to $u$ in $\Sch'_{\omega}$ and, hence,
\beqs
\label{Ad15}
\int_{\R^{2d}}K(z',z)V_\varphi u_n(z)dz=V_\varphi a(x,D) u_n(z')
\longrightarrow V_\varphi a(x,D)u(z')\qquad\mbox{in}\ \Sch'_\omega(\R^{2d}).
\eeqs

We want to
prove that
\beqs
\label{Ad14}
\int_{\R^{2d}}K(z',z)V_\varphi u_n(z)dz\longrightarrow
\int_{\R^{2d}}K(z',z)V_\varphi u(z)dz
\eeqs
using Lebesgue's dominated convergence theorem. First, it is easy to see that $\{V_{\varphi}u_{n}(z)\}_{n\in\N}$ converges pointwise to $V_{\varphi}u(z)$ for every $z\in\R^{2d}$ from the definition of the short-time Fourier transform.

Now, since $\{u_n\}_{n\in\N}$ is bounded in $\Sch'_\omega(\R^d)$, it is
  equicontinuous there. So, there exist a constant $C>0$ and a seminorm $q$ on $\Sch_\omega(\R^d)$
such that
\beqsn
|\langle u_n,\varphi\rangle|\leq C q(\varphi),
\qquad\varphi\in\Sch_\omega(\R^d).
\eeqsn
This yields a uniform estimate of the inequality \eqref{Ad7} (see the proof of
 \cite[Theorem 2.4]{GZ}) in the sense:
\beqs
\label{Ad12}
|V_\varphi u_n(z)|\leq \tilde{C}e^{\tilde{\lambda}\omega(z)},
\qquad z\in\R^{2d},\  n\in\N,
\eeqs
for some $\tilde{C},\tilde{\lambda}>0$ independent of $n$ and $z$. From \eqref{Ad12} and \eqref{Ad17}
we have that $K(z',z)V_\varphi u_n(z)$ is dominated by a function in
$L^1(\R^{2d}_z)$.

Therefore \eqref{Ad14} is satisfied and hence, from \eqref{Ad15},
\beqsn
V_\varphi a(x,D)u(z')=\int_{\R^{2d}}K(z',z)V_\varphi u(z)dz
\eeqsn
also for $u\in\Sch'_\omega(\R^d)$.
\end{proof}

We recall the notion of conic support from \cite{RW}:
\begin{Def}
\label{def21RW}
For $a\in\D'(\R^{2d})$ the {\em conic support} of $a$, denoted by $\conesupp(a)$, is the
set of all $z\in\R^{2d}\setminus\{0\}$ such that any open conic set $\Gamma\subset\R^{2d}\setminus\{0\}$ containing $z$ satisfies that
\beqsn
\overline{\supp(a)\cap\Gamma}\ \mbox{is not compact in $\R^{2d}$}.
\eeqsn
\end{Def}

We have the following
\begin{Prop}
\label{cor39RW}
If $m\in\R$, $a\in S^m_\omega$ and $u\in\Sch'_\omega(\R^d)$, then
\beqsn
\WF'_\omega(a(x,D)u)\subseteq\conesupp(a).
\eeqsn
\end{Prop}

\begin{proof}
Let $0\neq z_0\notin\conesupp(a)$.
This means that there exists an open conic set $\Gamma\subset\R^{2d}\setminus\{0\}$
containing $z_0$ and such that $a(z)=0$ for $z\in\Gamma\setminus\overline{B(0,R)}$
for some $R>0$.
Then, from Proposition~\ref{prop37RW}, for every open conic set $\Gamma'\subseteq\R^{2d}\setminus\{0\}$ with $\overline{\Gamma'\cap S_{2d-1}}\subseteq\Gamma$
we have that the kernel $K(z',z)$ defined by \eqref{312RW}
satisfies the estimate \eqref{314RW} for all $z'\in\Gamma'$ and $z\in\R^{2d}$.

We argue as in Corollary~\ref{corkernel} and use  \eqref{314RW} to obtain that formula \eqref{Ad18} holds for all $z'\in\Gamma'$
and therefore
there exist $C,\bar\lambda>0$ and for every $\lambda,N>0$ there exists
$C_{\lambda,N}>0$
such that, for all $z'\in\Gamma'$,
\beqsn
|V_\varphi(a(x,D)u)(z')|\leq&&\int_{\R^{2d}}|K(z',z)|\cdot|V_\varphi u(z)|dz\\
\leq&&C_{\lambda,N} e^{-2(\lambda+N)\omega(y')}e^{-2(\lambda+N)\omega(\eta')}\\
&&\cdot
\int_{\R^{2d}}e^{-(\lambda+N)\omega(y-y')}e^{-(\lambda+N)\omega(\eta-\eta')}
|V_\varphi u(y,\eta)|dyd\eta\\
\leq&&CC_{\lambda,N} e^{-2(\lambda+N)\omega(y')}e^{-2(\lambda+N)\omega(\eta')}\\
&&\cdot
\int_{\R^{2d}}e^{-(\lambda+N)\omega(y-y')}e^{-(\lambda+N)\omega(\eta-\eta')}
e^{\bar\lambda\omega(y,\eta)}dyd\eta.
\eeqsn

It follows, by the subadditivity of $\omega$, that
\beqs
\nonumber
|V_\varphi a(x,D)u(z')|\leq&&CC_{\lambda,N}
e^{-2(\lambda+N)\omega(y')}e^{-2(\lambda+N)\omega(\eta')}\\
\nonumber
&&\cdot
\int_{\R^{2d}}e^{-(\lambda+N)\omega(y)+(\lambda+N)\omega(y')}
e^{-(\lambda+N)\omega(\eta)+(\lambda+N)\omega(\eta')}
e^{\bar\lambda\omega(y)+\bar\lambda\omega(\eta)}dyd\eta\\
\label{39}
\leq&&CC_{\lambda,N} e^{-\lambda\omega(y')}e^{-\lambda\omega(\eta')}
\int_{\R^{2d}}e^{(\bar\lambda-N)\omega(y)}e^{(\bar\lambda-N)\omega(\eta)}
dyd\eta\\
\nonumber
\leq&&
C_\lambda e^{-\lambda\omega(y')}e^{-\lambda\omega(\eta')}
\leq C_\lambda e^{-\lambda\omega(z')}
\qquad\forall z'=(y',\eta')\in\Gamma'
\eeqs
for some $C_\lambda>0$ if we choose $N$ sufficiently large so that the
integral in \eqref{39} converges.

This proves that $z_0\notin\WF'_\omega(a(x,D)u)$ by Definition~\ref{def31RW}, and
the proof is complete.
\end{proof}

Since our weight functions are non-quasianalytic, we can obtain the following  consequence of Proposition~\ref{cor39RW}.
\begin{Cor}
Let $a\in\Sch_{\omega}(\R^{2d})$ with compact support, and consider the corresponding pseudo-differential operator $a(x,D)$, cf. \eqref{32}. Then $a(x,D)$ is globally $\omega$-regularizing, in the sense that for every $u\in\Sch^\prime_\omega(\R^d)$ we have
$a(x,D)u\in\Sch_\omega(\R^d)$.
\end{Cor}
\begin{proof}
It is easy to see that $a\in S^0_\omega$. Consequently, the corresponding pseudo-differential operator $a(x,D)$ can be extended to $\Sch^\prime_\omega(\R^d)$.  Since the support of $a$ is compact, we have that $\conesupp (a)=\emptyset$. From Proposition~\ref{cor39RW} we get
$\WF'_\omega (a(x,D)u)=\emptyset.$ We apply  Proposition~\ref{WFSomega} to conclude.
\end{proof}

In the next part of the section we consider other kind of operators, proving that they do not enlarge the wave front set. We start from the operators with polynomial coefficients.
\begin{Th}
\label{WFpolynom}
Let $m>0$ be an integer, and consider
\beqsn
A(x,D)= \sum_{\vert\alpha+\beta\vert\leq m} c_{\alpha\beta} x^\alpha D^\beta_x,
\eeqsn
where $c_{\alpha\beta}\in\C$. Then for every $u\in\Sch_\omega^\prime(\R^d)$ we have
\beqsn
\WF'_\omega (A(x,D)u)\subseteq \WF'_\omega(u).
\eeqsn
\end{Th}
\begin{proof}
We fix a window function $\varphi\in\Sch_\omega(\R^d)$, and, for $\nu\in\N^d_0$ we write $\varphi_\nu$ for the function
\beqsn
\varphi_\nu(x)=x^\nu\varphi(x).
\eeqsn
For every $\alpha\in\N^d_0$ and $z=(y,\eta)\in\R^{2d}$ we obtain by induction on $\vert\alpha\vert$ that
\beqs
\label{59}
x^\alpha \Pi(z)\varphi=\sum_{\nu\leq\alpha}\binom{\alpha}{\nu} y^{\alpha-\nu}\Pi(z)\varphi_\nu.
\eeqs
We have indeed that for $\vert\alpha\vert=1$, writing $\boldsymbol{1}_j$ for the multi-index in $\N^d_0$ having $1$ in the $j$-th position and $0$ elsewhere, we have
\beqsn
x_j\Pi(z)\varphi = y_j\Pi(z)\varphi + \Pi(z)\varphi_{\boldsymbol{1}_j};
\eeqsn
we suppose now that \eqref{59} is true for every $\vert\alpha\vert=n$, and prove it for $\tilde{\alpha}$ with $\vert\tilde{\alpha}\vert=n+1$. There exists $j\in\{1,\dots,d\}$ such that $\tilde{\alpha}=\alpha+\boldsymbol{1}_j$. Then by the inductive hypothesis we have
\beqsn
x^{\tilde{\alpha}} \Pi(z)\varphi &=& x_j\sum_{\nu\leq\alpha}\binom{\alpha}{\nu} y^{\alpha-\nu}\Pi(z) \varphi_\nu \\
&=&\sum_{\nu\leq\alpha}\binom{\alpha}{\nu} \left[ y^{\alpha-\nu+\boldsymbol{1}_j}\Pi(z)\varphi_\nu + y^{\alpha-\nu}\Pi(z) \varphi_{\nu+\boldsymbol{1}_j}\right] \\
&=& y^{\tilde{\alpha}}\Pi(z)\varphi + \Pi(z)\varphi_{\tilde{\alpha}} +\sum_{\substack{\nu\leq\alpha \\ \nu\neq 0}}\left[ \binom{\alpha}{\nu}+\binom{\alpha}{\nu-\boldsymbol{1}_j}\right] y^{\tilde{\alpha}-\nu}\Pi(z) \varphi_\nu \\
&=& \sum_{\nu\leq\tilde{\alpha}} \binom{\tilde{\alpha}}{\nu} y^{\tilde{\alpha}-\nu} \Pi(z)\varphi_\nu,
\eeqsn
and so \eqref{59} is proved. From the definition of short-time Fourier transform we have
\beqsn
V_\varphi(x^\alpha u)(z) = \langle x^\alpha u,\Pi(z)\varphi\rangle = \langle u,x^\alpha \Pi(z)\varphi\rangle
\eeqsn
and so by \eqref{59} we get
\beqs
\label{60}
V_\varphi(x^\alpha u)(z) = \sum_{\nu\leq\alpha} \binom{\alpha}{\nu} y^{\alpha-\nu} V_{\varphi_\nu} u(z).
\eeqs
Concerning derivation, since
\beqsn
V_\varphi(D^\beta u)(z) = \langle D^\beta u,\Pi(z)\varphi\rangle = \langle u,D^\beta(\Pi(z)\varphi)\rangle
\eeqsn
a direct computation shows that
\beqs
\label{61}
V_\varphi(D^\beta u)(z) = \sum_{\mu\leq\beta}\binom{\beta}{\mu} \eta^{\beta-\mu} V_{D^\mu\varphi} u.
\eeqs
From \eqref{60} and \eqref{61} we finally obtain
\beqs
V_\varphi(A(x,D)u)(y,\eta) &=& \sum_{\vert\alpha+\beta\vert\leq m} c_{\alpha\beta} V_\varphi (x^\alpha D^\beta_x u)(y,\eta) \notag\\
&=& \sum_{\vert\alpha+\beta\vert\leq m} \sum_{\substack{\nu\leq\alpha \\ \mu\leq\beta}} c_{\alpha\beta} \binom{\alpha}{\nu}\binom{\beta}{\mu} y^{\alpha-\nu} \eta^{\beta-\mu} V_{D^\mu \varphi_\nu} u(y,\eta).
\label{62}
\eeqs
On the other hand, it is not difficult to see  that for every $\mu,\nu\in\N^d_0$, $D^\mu\varphi_\nu\in\Sch_\omega(\R^d)$.

Suppose now that $z_0=(y_0,\eta_0)\notin \WF'_\omega(u)$, $z_0\in\R^{2d}\setminus\{0\}$. Then, there exists an open conic set $\Gamma\subseteq\R^{2d}\setminus\{0\}$ containing $z_0$ and such that
\beqsn
\sup_{z\in\Gamma} e^{\lambda\omega(z)}\vert V_\varphi u(z)\vert<+\infty, \qquad \lambda>0.
\eeqsn
From Proposition~\ref{cor33RW} we have that for every $\mu,\nu\in\N^d_0$ and for every open conic set $\Gamma^\prime\subseteq\R^{2d}\setminus\{0\}$ containing $z_0$ and such that $\overline{\Gamma^\prime\cap S_{2d-1}}\subseteq\Gamma$,
\beqs
\label{63}
\sup_{z\in\Gamma^\prime} e^{\lambda\omega(z)}\vert V_{D^\mu \varphi_\nu} u(z)\vert<+\infty \qquad \forall\lambda>0.
\eeqs
From \eqref{62}, for every $k>0$ we get
\beqsn
e^{\lambda\omega(z)}\vert V_\varphi (A(x,D)u)(z)\vert\leq \sum_{\vert\alpha+\beta\vert\leq m} \sum_{\substack{\nu\leq\alpha \\ \mu\leq\beta}} c_{\alpha\beta} \binom{\alpha}{\nu}\binom{\beta}{\mu} e^{-k\omega(z)} \vert y^{\alpha-\nu} \eta^{\beta-\mu}\vert e^{(\lambda+k)\omega(z)} \vert V_{D^\mu \varphi_\nu} u(z)\vert.
\eeqsn
Since $\vert\alpha-\nu\vert+\vert\beta-\mu\vert\leq m$, from the property $(\gamma)$ of the weight function $\omega$ we obtain
\beqsn
\sup_{z\in\R^{2d}} e^{-k\omega(z)} \vert y^{\alpha-\nu} \eta^{\beta-\mu}\vert <+\infty,
\eeqsn
for every $\nu\leq\alpha$, $\mu\leq\beta$. Therefore,  from \eqref{63} we obtain
\beqsn
\sup_{z\in\Gamma^\prime}e^{\lambda\omega(z)}\vert V_\varphi (A(x,D)u)(z)\vert<+\infty, \qquad\lambda>0,
\eeqsn
which means that $z_0\notin \WF'_\omega(A(x,D)u)$, and  the proof is complete.
\end{proof}

We now want to prove an analogue of Theorem~\ref{WFpolynom} for the case of localization operators. We recall here the definition of such operators and prove some results that are needed for our purpose. Given two window functions $\psi,\gamma\in\Sch_\omega(\R^d)\setminus\{0\}$ and a symbol $a\in\Sch^\prime_\omega(\R^{2d})$, the corresponding localization operator $L^a_{\psi,\gamma}$ is defined, for $f\in\Sch_\omega(\R^d)$, as
\beqs
\label{64}
L^a_{\psi,\gamma} f =V^*_\gamma(a \cdot V_\psi f).
\eeqs
From Proposition~\ref{propcontS} we have that
\beqsn
L^a_{\psi,\gamma}:\Sch_\omega(\R^d) \rightarrow \Sch^\prime_\omega(\R^d).
\eeqsn
We want now to consider symbols in a smaller class than $\Sch^\prime_\omega(\R^{2d})$, in order to apply the corresponding localization operator to distributions. We have the following result.
\begin{Lemma}
\label{contlocaliz}
Let $a(z)$, $z\in\R^{2d}$, be a measurable function such that there exist $\tau,C>0$ such that
\beqs
\label{65}
\vert a(z)\vert\leq C e^{\tau\omega(z)} \qquad\forall z\in\R^{2d}.
\eeqs
Then
\beqs
\label{66}
L^a_{\psi,\gamma}:\Sch_\omega(\R^d) \rightarrow \Sch_\omega(\R^d)
\eeqs
and
\beqs
\label{67}
L^a_{\psi,\gamma}:\Sch^\prime_\omega(\R^d) \rightarrow \Sch^\prime_\omega(\R^d)
\eeqs
are continuous.
\end{Lemma}
\begin{proof}
Let $f\in\Sch_\omega(\R^d)$. From Theorem~\ref{th27GZ} we have that for every $\lambda,\rho>0$ there exists $C_\lambda>0$ such that
\beqsn
e^{\rho\omega(z)}\vert a(z)\vert \vert V_\psi f(z)\vert\leq C_\lambda e^{(\rho+\tau-\lambda)\omega(z)},
\eeqsn
and so, choosing $\lambda\geq \rho+\tau$, we have that $a\cdot V_\psi f\in L^\infty_{m_\rho}(\R^{2d})$ for every $\rho>0$, where $m_\rho$ is defined by \eqref{21bis}. From Proposition~\ref{prop1132G} and \eqref{64}, we have that $L^a_{\psi,\gamma}f\in \boldsymbol{M}^\infty_{m_\rho}(\R^d)$ for every $\rho>0$, and then, from Remark~\ref{form25RW}, $L^a_{\psi,\gamma} f\in\Sch_\omega(\R^d)$. To prove the continuity of $L^a_{\psi,\gamma}$ on $\Sch_\omega(\R^d)$ we fix $\varphi\in\Sch_\omega(\R^d)\setminus\{0\}$, $\rho>0$, and we observe that from \eqref{68} (with $p=q=\infty$) and \eqref{65} we get
\beqsn
\sup_{z\in\R^{2d}} \vert V_\varphi (L^a_{\psi,\gamma}f)(z)\vert e^{\rho\omega(z)} &=& \sup_{z\in\R^{2d}} \vert V_\varphi V^*_\gamma(a\cdot V_\psi f)\vert e^{\rho\omega(z)} \\
&\leq& C\Vert V_\varphi\gamma\Vert_{L^1_{v_\rho}}\sup_{z\in\R^{2d}} \vert a(z) V_\psi f(z)\vert e^{\rho\omega(z)} \\
&\leq& C^\prime \sup_{z\in\R^{2d}} \vert V_\psi f(z)\vert e^{(\tau+\rho)\omega(z)}.
\eeqsn
From Proposition~\ref{cor1126G} we have that \eqref{66} is continuous. \hfill\break
Let now $f\in\Sch_\omega^\prime(\R^d)$. From Remark~\ref{form25RW} there exists $\lambda<0$ such that $f\in\boldsymbol{M}^\infty_{m_\lambda}(\R^d)$; then, choosing $\rho=-\vert\tau\vert-\vert\lambda\vert$ we have
\beqsn
e^{\rho\omega(z)}\vert a(z)\vert \vert V_\psi f(z)\vert\leq C e^{(\rho+\tau-\lambda)\omega(z)}<+\infty
\eeqsn
for every $z\in\R^{2d}$, so $a\cdot V_\psi f\in L^\infty_{m_\rho}(\R^{2d})$. Then by Proposition~\ref{prop1132G} we have $L^a_{\psi,\gamma}f\in\boldsymbol{M}^\infty_{m_\rho}(\R^d)$, and from Remark~\ref{form25RW} we finally have $L^a_{\psi,\gamma}f\in\Sch^\prime_\omega(\R^d)$. Observe now that for every $u\in\Sch_\omega^\prime(\R^d)$ and $v\in\Sch_\omega(\R^d)$ we have
\beqsn
\langle L^a_{\psi,\gamma} u,v\rangle = \langle V^*_\gamma(a\cdot V_\psi u),v\rangle = \langle u,V^*_\psi(\overline{a}\cdot V_\gamma v)\rangle = \langle u,L^{\overline{a}}_{\gamma,\psi}v\rangle.
\eeqsn
Then $L^a_{\psi,\gamma}=(L^{\overline{a}}_{\gamma,\psi})^*$; since $\overline{a}$ satisfies the same estimates as $a$, the continuity of \eqref{67} follows from that of \eqref{66}.
\end{proof}

\begin{Th}
Let $\psi,\gamma\in\Sch_\omega(\R^d)\setminus\{0\}$, and let $a$ be a symbol satisfying \eqref{65}. Then for every $u\in\Sch^\prime_\omega(\R^d)$ we have
\beqsn
\WF'_\omega(L^a_{\psi,\gamma}u)\subseteq \WF'_\omega(u).
\eeqsn
\end{Th}
\begin{proof}
Let $z_0\notin \WF'_\omega(u)$, $z_0\in\R^{2d}\setminus\{0\}$. Then there exists an open conic set $\Gamma\subseteq\R^{2d}\setminus\{0\}$ containing $z_0$ such that
\beqsn
\sup_{z\in\Gamma} e^{\lambda\omega(z)}\vert V_\psi u(z)\vert <+\infty \qquad\forall\lambda>0.
\eeqsn
From \eqref{65}, since $\lambda$ is arbitrary we have
\beqsn
\sup_{z\in\Gamma} e^{\lambda\omega(z)}\vert a(z) V_\psi u(z)\vert <+\infty \qquad\forall\lambda>0.
\eeqsn
For windows functions $\varphi,\gamma\in\Sch_\omega(\R^d)$ we can then repeat the same procedure used in the proof of Proposition~\ref{cor33RW}. First, we observe that from de definition of localization operator
$$
V_{\varphi}(L^{a}_{\psi,\gamma}u)=V_{\varphi}V^{*}_{\gamma}(a\cdot V_{\psi}u).
$$
Now, it is not difficult to see that
\beqsn
&&V_\varphi(L^a_{\psi,\gamma}u)(x,\xi)=\int_{\R^{2d}}(a\cdot V_\psi u)(s,\eta)
\overline{V_\gamma(\Pi(z)\varphi)}(s,\eta)dsd\eta,\\
&&
\overline{V_\gamma(\Pi(z)\varphi)}(s,\eta)=V_\varphi \gamma(x-s,\xi-\eta)
e^{-i\langle s,\xi-\eta\rangle},
\eeqsn
and hence
$$
|V_{\varphi}(L^{a}_{\psi,\gamma}u)|\le |a\cdot V_{\psi}u|*|V_{\varphi}\gamma|.
$$
Consequently, for every open conic set $\Gamma^\prime\subseteq\R^{2d}\setminus\{0\}$ containing $z_0$ and such that $\overline{\Gamma^\prime\cap S_{2d-1}}\subseteq\Gamma$ we have (see the proof of Proposition~\ref{cor33RW})
\beqsn
\sup_{z\in\Gamma^\prime} e^{\lambda\omega(z)} \vert V_\varphi(L^a_{\psi,\gamma}u)(z)\vert <+\infty,\qquad\lambda>0.
\eeqsn
This implies that $z_0\notin \WF'_\omega(L^a_{\psi,\gamma}u)$ and the proof is complete.
\end{proof}

\section{Examples}
\label{sec5}

In this section we compute the Gabor wave front set for some particular $u\in\Sch^\prime_\omega(\R^d)$ (see also the examples in \cite{RW}).
\begin{Ex}
\label{delta}{\rm
Consider the Dirac distribution $u=\delta\in\Sch_\omega^\prime(\R^d)$ for every weight $\omega$. We have 
that
\beqsn
V_\varphi\delta(x,\xi)=\overline{\varphi(-x)}.
\eeqsn
Since $V_\varphi\delta(0,\xi)=\overline{\varphi(0)}$, choosing $\varphi$ in such a way that $\varphi(0)\neq 0$ we have
\beqsn
\{0\}\times(\R^d\setminus\{0\})\subseteq \WF'_\omega(\delta).
\eeqsn
Let now $(x_0,\xi_0)\in\R^{2d}\setminus\{0\}$ such that $x_0\neq 0$, and consider an open conic set containing $(x_0,\xi_0)$ of the form
\beqsn
\Gamma=\{(x,\xi)\in\R^{2d}\setminus\{0\} : \vert\xi\vert < C\vert x\vert\}
\eeqsn
for $C>0$. From the subadditivity of $\omega$, there exists $C_1>0$ such that, writing $z=(x,\xi)$,
\beqsn
\sup_{z\in\Gamma} e^{\lambda\omega(z)} \vert V_\varphi \delta(z)\vert\leq \sup_{x\in\R^d} e^{\lambda C_1\omega(x)} \vert\varphi(-x)\vert <+\infty
\eeqsn
since $\varphi\in\Sch_\omega(\R^d)$. Then $(x_0,\xi_0)\notin \WF'_\omega(\delta)$, and so $\WF'_\omega(\delta)=\{0\}\times(\R^d\setminus\{0\})$. From Proposition~\ref{cor52RW} we have that for every $\overline{x}\in\R^d$, writing $\delta_{\overline{x}}$ for the Dirac distribution centered at $\overline{x}$,
\beqs
\label{69}
\WF'_\omega(\delta_{\overline{x}})=\{0\}\times(\R^d\setminus\{0\}).
\eeqs}
\end{Ex}

\begin{Ex}{\rm
Let $u=\boldsymbol{1}$ be the function identically $1$, that belong to $\Sch^\prime_\omega(\R^d)$ for every weight $\omega$. A direct computation shows that
\beqsn
V_\varphi(\boldsymbol{1})=e^{-i\langle x,\xi\rangle} \overline{\hat{\varphi}(-\xi)};
\eeqsn
since $\hat{\varphi}\in\Sch_\omega(\R^d)$ we can proceed as in Example~\ref{delta}, obtaining that for every weight $\omega$, $\WF'_\omega(\boldsymbol{1}) = (\R^d\setminus\{0\})\times\{0\}$. From Proposition~\ref{cor52RW} we then have that for every $\overline{\xi}\in\R^d$ and for every weight $\omega$,
\beqs
\label{70}
\WF'_\omega(e^{i\langle \cdot,\overline{\xi}\rangle}) = (\R^d\setminus\{0\})\times\{0\}.
\eeqs}
\end{Ex}

\begin{Ex}{\rm
We consider now the function $u(x)=e^{icx^2/2}$, for $x\in\R$ and $c\in\R\setminus\{0\}$. Observe that $u\in\Sch^\prime_\omega(\R)$ for every $\omega$. Choosing as window function the Gaussian $\varphi(t)=e^{-t^{2}/2}$, that belongs to $\Sch_\omega(\R)$ for every $\omega$, we have, as in Example 6.6 of \cite{RW}, that there exists $C>0$ such that
\beqsn
\vert V_\varphi u(x,\xi)\vert = C\exp\left( -\frac{(\xi-c x)^2}{2(1+c^2)}\right).
\eeqsn
Then, proceeding in a similar way as in the previous cases we have
\beqs
\label{71}
\WF'_\omega(u)=\{(x,cx):x\in\R\setminus\{0\}\}
\eeqs
for every weight $\omega$.}
\end{Ex}

We observe that in the cases \eqref{69} and \eqref{70} the Gabor wave front set gives rougher information since it does not take into account translations and modulations, while for the case \eqref{71} it gives finer information, since it identifies the so-called {\em instantaneous frequency}, that is the only direction along which the time-frequency content of $u$ does not decay. For a comparison of the Gabor wave front set of the element considered in the previous examples with other type of global wave front set (at least in the frame of tempered distributions) we refer to \cite{RW}.

We observe now that in the previous examples the considered distributions have the same wave front set for every weight $\omega$. In general the Gabor wave front set may depend on $\omega$, as shown in the next example.
\begin{Ex}{\rm
Let $\omega$ and $\sigma$ be two weight functions, such that $\omega(t)\leq\sigma(t)$ and $\mathcal{S}_\sigma(\R^d)\cap\D(\R^d)\subsetneq \mathcal{S}_\omega(\R^d)\cap\D(\R^d)$. We then fix a function $f\in \mathcal{S}_\omega(\R^d)$ with compact support such that $f\notin \mathcal{S}_\sigma(\R^d)$.
From Proposition~\ref{WFSomega} we have
\beqsn
\WF'_\omega(f)=\emptyset.
\eeqsn
Fix now a window $\varphi_0\in\Sch_\sigma(\R^d)$ with compact support such that $\varphi_0\equiv 1$ on $\supp(f)$. From the definition of short-time Fourier transform, we then have that the orthogonal projection on $\R^d_x$ of the support of $V_{\varphi_0}f(x,\xi)$ is compact. Let now $z_0=(x_0,\xi_0)\in\R^{2d}$ with $x_0\neq 0$, and fix an open conic set containing $z_0$ of the form
\beqsn
\Gamma=\{(x,\xi)\in\R^{2d}\setminus\{0\}:\vert\xi\vert < C\vert x\vert\},
\eeqsn
for $C>0$. We then have that $\overline{\Gamma\cap\supp(V_{\varphi_0}f)}$ is compact, so the condition \eqref{16} is satisfied for every $\lambda>0$. Then $(x_0,\xi_0)\notin\WF'_\sigma(f)$ for every $x_0\neq 0$. Consider now a point of the type $(0,\xi_0)$ with $\xi_0\neq 0$, $\xi_0\in\R^d$. From the fact that $\varphi_0\equiv 1$ on $\supp(f)$, we have
\beqsn
V_{\varphi_0}f(0,\xi)=\int e^{-i\langle t,\xi\rangle} f(t) \overline{\varphi_0 (t)}\,dt = \hat{f}(\xi).
\eeqsn
Since $f\notin \mathcal{S}_\sigma(\R^d)$, we have that there exists $\lambda>0$ such that
\beqsn
\sup_{\xi\in\R^d} e^{\lambda\sigma(\xi)} \vert V_{\varphi_0} f(0,\xi)\vert=+\infty,
\eeqsn
so \eqref{16} cannot be satisfied in an open conic set containing $(0,\xi_0)$, and then $(0,\xi_0)\in \WF'_\omega(f)$. We then have that
\beqsn
\WF'_\sigma(f)=\{0\}\times(\R^d\setminus\{0\});
\eeqsn
in particular $\WF'_\sigma(f)\neq \WF'_\omega(f)$.}
\end{Ex}


\vspace*{10mm}
{\bf Acknowledgments.}
The authors were partially supported by the INdAM-Gnampa Project 2016 ``Nuove prospettive
nell'analisi microlocale e tempo-frequenza'', by FAR 2013 (University of Ferrara) and by the project ``Ricerca Locale - Analisi di Gabor, operatori pseudodifferenziali ed equazioni differenziali'' (University of Torino). The research of the second author was partially supported by the project  MTM2016-76647-P


\begin{thebibliography}{AAA}

\bibitem[AJO-1]{AJO1}
A.~Albanese, D.~Jornet, A.~Oliaro, {\em Quasianalytic wave front sets
for solutions of linear partial differential operators}, Integr. Equ.
Oper. Theory {\bf 66}  (2010), 153-181.

\bibitem[AJO-2]{AJO2}
A.~Albanese, D.~Jornet, A.~Oliaro, {\em Wave front sets for ultradistribution solutions of linear partial differential operators with coefficients in non-quasianalytic classes},
Math. Nachr. {\bf 285}, n. 4 (2012), 411-425.

\bibitem[B]{B}
G.~Bj\"orck, {\em Linear partial differential operators and generalized
distributions}, Ark. Mat. {\bf 6}, n. 21 (1966), 351-407.

\bibitem[BG]{BG}
C.~Boiti, E.~Gallucci, {\em The overdetermined Cauchy problem for
$\omega$-ultradifferentiable functions}, Manuscripta Math. (2017)
DOI: 10.1007/s00229-017-0939-2

\bibitem[BJ]{BJ-Kotake}
C.~Boiti, D.~Jornet, {\em A simple proof of Kotake-Narasimhan Theorem in
some classes of
ultradifferentiable functions}, J. Pseudo-Differ. Oper. Appl. 8(2), (2017), 297-317.

\bibitem[BJJ]{BJJ}
C.~Boiti, D.~Jornet, J. Juan-Huguet, {\em Wave front sets with respect to the iterates of an operator with constant coefficients},
Abstr. Appl. Anal., (2014), Art. ID 438716, 1-17; http://dx.doi.org/10.1155/2014/438716

\bibitem[BJO]{BJO}
C.~Boiti, D.~Jornet, A.~Oliaro, {\em Regularity of partial differential
operators in ultradifferentiable spaces and Wigner type transforms},
J. Math. Anal. Appl. {\bf 446} (2017),  920-944.

\bibitem[BMM]{BMM} J.~Bonet, R.~Meise, S.N.~Melikhov, {\em A comparison of two different ways to define classes of ultradifferentiable functions},    Bull. Belg. Math. Soc. Simon Stevin, Volume 14, Number 3 (2007), 425-444.

\bibitem[BMT]{BMT}
R.W.~Braun, R.~Meise, B.A.~Taylor,
{\em Ultradifferentiable functions and Fourier analysis}, Result. Math.
{\bf 17} (1990), 206--237.

\bibitem[CS]{CS}
M.~Cappiello, R.~Schulz, {\em Microlocal analysis of quasianalytic Gelfand-Shilov type ultradistributions},
Complex Var. Elliptic Equ. {\bf 61}, n. 4 (2016), 538-561.

\bibitem[CW]{CW}
E.~Carypis, P.~Wahlberg, {\em Propagation of Exponential Phase Space Singularities for Schr\"odinger Equations with Quadratic Hamiltonians},
J. Fourier Anal. Appl. {\bf 23}, n. 3 (2017), 530-571.

\bibitem[FGJ]{FGJ}
C.~Fern\'andez, A.~Galbis, D.~Jornet, {\em Pseudodifferential operators on non-quasianalytic classes of Beurling type}, Studia Math. {\bf 167}, n. 2 (2005), 99-131.

\bibitem[G]{G}
K.~Gr\"ochenig, {\em Foundations of Time-Frequency Analysis},
Birkh\"auser, Boston  (2001).

\bibitem[GZ]{GZ}
K.~Gr\"ochenig, G.~Zimmermann, {\em Spaces of Test Functions via the STFT},
J. Funct. Spaces Appl. {\bf 2}, n. 1 (2004), 25-53

\bibitem[H-1]{H}
L.~H\"ormander, {\em Fourier integral operators},
Acta Math. {\bf 127}, n. 1 (1971), 79-183.

\bibitem[H-2]{H2}
L.~H\"ormander, {\em Quadratic Hyperbolic operators}. In Cattabriga, L., Rodino, L. (eds.) Microlocal Analysis and Applications. Lecture Notes in Mathematics, vol. 1495, pp. 118-160. Springer, Berlin (1991).

\bibitem[H-3]{H3}
L.~H\"ormander, {\em The Analysis of Linear Partial Differential Operators I},
Grundlehren der mathematischen Wissenschaften, vol. 256, Springer-Verlag, Berlin, 1990.

\bibitem[J]{J}
A.J.E.M. Janssen, {\em Duality and Biorthogonality for Weyl-Heisenberg Frames},
J. Fourier Anal. Appl. {\bf 1}, n. 4 (1995), 403-436.

\bibitem[L]{L}
M.~Langenbruch, {\em Hermite functions and weighted spaces of generalized functions},
Manuscripta Math. {\bf 119}, n. 3 (2006), 269-285.

\bibitem[MV]{MV}
R.~Meise, D.~Vogt, {\em Introduction to Functional Analysis}, Oxford Science
Publications, Clarendon Press, Oxford, 1997.

\bibitem[N]{N}
S.~Nakamura, {\em Propagation of the homogeneous wave front set for Schr\"odinger equations},
Duke Math. J. {\bf 126} (2005), 349-367.

\bibitem[NR]{NR}
F.~Nicola, L.~Rodino, {\em Global Pseudo-Differential
Calculus on Euclidean Spaces}, Springer, Basel (2010).

\bibitem[R]{R}
L.~Rodino, {\em Linear partial differential operators and Gevrey spaces}, World Scientific Publishing Co., Inc. River Edge, NJ, (1993).

\bibitem[RW]{RW}
L.~Rodino, P.~Wahlberg, {\em The Gabor wave front set}, Monatsh. Math. {\bf 173}
(2014), 625-655.

\bibitem[SW-1]{SW1}
R.~Schulz, P.~Wahlberg, {\em Microlocal properties of Shubin pseudodifferential and localization operators},
J. Pseudo-Differ. Oper. Appl. {\bf 7}, n. 1 (2016), 91-111.

\bibitem[SW-2]{SW2}
R.~Schulz, P.~Wahlberg, {\em Equality of the homogeneous and the Gabor wave front set},
Comm. Partial Differential Equations {\bf 42}, n. 5 (2017), 703-730.

\bibitem[S]{S}
M.A.~Shubin, {\em Pseudodifferential Operators and
Spectral Theory}, Sprin\-ger-Verlag, Berlin, 1987.

\bibitem[T-1]{T-1}
J.~Toft, {\em The Bargmann transform on modulation
and Gelfand-Shilov spaces, with applications
to Toeplitz and pseudo-differential operators}, J. Pseudo-Differ. Oper. Appl. {\bf 3}, n.2
(2012), 145-227.

\bibitem[T-2]{T-2}
J.~Toft, {\em Images of function and distribution spaces under
the Bargmann transform}, J. Pseudo-Differ. Oper. Appl. {\bf 8}, n.1 (2017), 83-139.

\bibitem[T]{T}
F.~Treves, {\em Topological vector spaces, distributions and kernels},
Academic Press, New York, 1967.

\end{thebibliography}
\end{document}